\documentclass[12pt]{amsart}
\usepackage[dvips]{color}
\usepackage{amsmath}
\usepackage{amsxtra}
\usepackage{amscd}
\usepackage{amsthm}
\usepackage{amsfonts}
\usepackage{amssymb}
\usepackage{eucal}
\usepackage{epsfig}
\usepackage{graphics}
\usepackage{accents}
\numberwithin{equation}{section}
\newtheorem{thm}{Theorem}[section]
\newtheorem{prop}[thm]{Proposition}
\newtheorem{lem}[thm]{Lemma}
\newtheorem{rem}[thm]{Remark}
\newtheorem{cor}[thm]{Corollary}

\newcommand{\nn}{\nonumber}

\newcommand{\ket}[1]{{| #1 \rangle}}      
\newcommand{\br}[1]{{\langle #1 \rangle}}  
\newcommand{\ds}[1]{\displaystyle #1}

\newcommand{\C}{{\mathbb C}}
\newcommand{\Z}{{\mathbb Z}}
\newcommand{\Q}{{\mathbb Q}}
\newcommand{\R}{{\mathbb R}}

\newcommand{\cO}{\mathcal{O}}

\newcommand{\cR}{\mathcal{R}}
\newcommand{\cT}{\mathcal{T}}

\newcommand{\Rb}{\overline{R}}

\newcommand{\bs}{\boldsymbol}
\newcommand{\bm}{{\mathfrak{m}}}
\newcommand{\bn}{{\mathfrak{n}}}

\newcommand{\gl}{\mathfrak{gl}}

\newcommand{\slt}{\mathfrak{sl}_2}
\newcommand{\slth}{\widehat{\mathfrak{sl}}_2}
\newcommand{\Tb}{\mathfrak{T}}

\newcommand{\la}{\lambda}

\newcommand{\sfq}{{\sf q}}

\newcommand{\End}{\mathop{\rm End}}
\newcommand{\Hom}{\mathop{\rm Hom}}
\newcommand{\id}{{\rm id}}

\newcommand{\Ob}{\mathrm{Ob}\,}
\newcommand{\Rep}{\mathrm{Rep}}
\newcommand{\res}{{\rm res}}

\newcommand{\Tr}{{\rm Tr}}

\newcommand{\wt}{{\rm wt}\,}
\newcommand{\gge}{\geqslant}

\newcommand{\on}{\operatorname}
\newcommand{\mc}{\mathcal}
\newcommand{\al}{\alpha}

\newcommand{\qbin}[2]{{\left[
\begin{matrix}{\,\displaystyle #1\,}\\
{\,\displaystyle #2\,}\end{matrix}
\right]
}}
\newcommand{\hdeg}{\mathop{\mathrm{hdeg}}}

\newcommand{\g}{\mathfrak{g}}
\newcommand{\gb}{\accentset{\circ}{\mathfrak{g}}}
\newcommand{\bo}{\mathfrak{b}}
\newcommand{\bbo}{\overline{\mathfrak{b}}}
\newcommand{\tb}{\mathbf{\mathfrak{t}}}

\newcommand{\U}{U}
\newcommand{\Psib}{\mbox{\boldmath$\Psi$}}
\newcommand{\Psibs}{\scalebox{.7}{\boldmath$\Psi$}}

\newcommand{\Fin}{\cO^{fin}_{\bo}}
\newcommand{\mfr}{\mathfrak{r}}

\newcommand{\supp}{\mathrm{supp}}

\definecolor{4.2}{rgb}{0,0,0}
\definecolor{zh}{rgb}{0.5,0.7,0.0}

\begin{document}

\begin{title}[Finite type modules and Bethe ansatz equations]
{Finite type modules and Bethe ansatz equations}
\end{title}
\author{B. Feigin, M. Jimbo, T. Miwa, and E. Mukhin}
\address{BF: National Research University Higher School of Economics, Russian Federation, International Laboratory of Representation Theory 
and \newline Mathematical Physics, Russia, Moscow,  101000,  Myasnitskaya ul., 20 and Landau Institute for Theoretical Physics,
Russia, Chernogolovka, 142432, pr.Akademika Semenova, 1a
}
\email{bfeigin@gmail.com}
\address{MJ: Department of Mathematics,
Rikkyo University, Toshima-ku, Tokyo 171-8501, Japan}
\email{jimbomm@rikkyo.ac.jp}
\address{TM: Institute for Liberal Arts and Sciences,
Kyoto University, Kyoto 606-8316,
Japan}\email{tmiwa@kje.biglobe.ne.jp}
\address{EM: Department of Mathematics,
Indiana University-Purdue University-Indianapolis,
402 N.Blackford St., LD 270,
Indianapolis, IN 46202, USA}\email{mukhin@math.iupui.edu}

\begin{abstract} 
We introduce and study a category $\Fin$ of modules  of the Borel subalgebra $U_q\bo$ of a quantum affine algebra $U_q\g$, 
where the commutative algebra
of Drinfeld generators $h_{i,r}$, corresponding to Cartan currents, has finitely many characteristic values. 
This category is a natural extension of the category of finite-dimensional  $U_q\g$ modules. In particular, 
we classify the irreducible objects, discuss their properties,
and describe the combinatorics of the $q$-characters. 
We study transfer matrices corresponding to modules in $\Fin$. Among them we find 
the Baxter $Q_i$ operators and  $T_i$ operators satisfying relations of the form $T_iQ_i=\prod_j Q_j+ \prod_k Q_k$. We show that 
these operators are polynomials of the spectral parameter after a suitable normalization.
This allows us to prove the Bethe ansatz equations for the zeroes of the eigenvalues of the $Q_i$ operators acting 
in an arbitrary finite-dimensional representation of $U_q\g$.
\end{abstract}
\date{\today}
\maketitle 

\section{Introduction}
The XXZ models are the celebrated integrable models whose Hamiltonians originate in quantum affine algebras.  

Let $U_q\g$ be a quantum affine algebra. The universal $R$ matrix is a special element $\mc R$ of the completion  
$U_q\g\widehat\otimes U_q\g$ which intertwines the standard comultiplication in $U_q\g$ with the opposite one. 
Given a $U_q\g$ module $V$, the transfer matrix is the trace $T_V=\on{Tr}_V(\pi_V\otimes \on{id})(\mc R)$. 
Due to the properties of the $R$ matrix (the Yang-Baxter equation), transfer matrices for various representations 
commute: $T_VT_W=T_WT_V$, and therefore give rise to a family of commuting operators in a completion of $U_q\g$. 
These operators act on any finite-dimensional $U_q\g$ module. They are called the XXZ Hamiltonians.

The XXZ models received a lot of attention after the pioneering work \cite{Ba} where the $U_q\mathfrak{sl}_2$ 
transfer matrix was studied in relation to the 6-vertex lattice model. The spectrum of the model is usually 
obtained via various forms of the Bethe ansatz. In the standard approach of algebraic or coordinate Bethe ansatz, 
one writes a certain vector, called wave function, depending on parameters. When the parameters satisfy a system 
of algebraic equations, called Bethe ansatz equations, the wave function becomes an eigenvector of the Hamiltonians. 
 \medskip

In this paper we describe the spectrum of the XXZ Hamiltonians using the so-called analytic Bethe ansatz. 
In particular, using this method, we deduce the Bethe ansatz equations from the absence of poles for a 
certain operator, as anticipated in \cite{R}, \cite{FR}, \cite{FH}. 

The crucial observation is that the $R$ matrix  is not only an element of the completed tensor square of 
$U_q\g$ but an element of $U_q\bo\widehat\otimes U_q\bbo$, where 
$U_q\bo$ and $U_q\bbo$ are the Borel subalgebras of $U_q\g$.
Therefore, in addition to transfer matrices $T_V$ where $V$ is a $U_q\g$ module, 
one can consider transfer matrices $T_V$ where $V$ is a $U_q\bo$ module. The algebra $U_q\bo$ has a significantly 
richer representation theory than that of $U_q\g$ and this additional freedom allows us to solve the problem.
\medskip

First, we develop the representation theory of $U_q\bo$.

In the category $\mc O_{\bo}$ of highest $\ell$-weight $U_q\bo$ modules introduced in \cite{HJ}, 
we define a subcategory $\Fin$ of modules of finite type. By definition, a module is of finite type 
if the joint spectrum of 
the Cartan generators $h_{i,r}$ (in the
Drinfeld realization of $U_q\g$) is finite. Here $i\in I$, 
 $I$ is the set of labels 
of simple roots for the finite-dimensional algebra corresponding to $\g$, and $r>0$. 

The positive fundamental modules  $M^+_{i,a}$, $i\in I$, $a\in\C^\times$, are irreducible modules 
with simplest possible highest $\ell$-weight given by 
$\Psi_j(z)=(a^{-1/2}-a^{1/2}z)\delta_{i,j}+(1-\delta_{i,j})$, 
$j\in I$. They first appeared in 
\cite{BLZ2} in the case $U_q\widehat{\mathfrak{sl}}_2$,  
and then were studied 
in \cite{BHK}, \cite{BJMST}, \cite{HJ}, \cite{BFLMS}, \cite{FH}. 
By \cite{FH}, the positive fundamental modules are of finite type. 
We prove that the category $\Fin$ is topologically
generated by finite-dimensional  
modules (which are essentially restrictions of $U_q\g$ modules) 
and the positive fundamental modules. 

The category $\Fin$ naturally extends the category of finite-dimensional $U_q\g$ modules and has 
similar properties. In particular, the theory of $q$-characters provides a powerful tool for the study. 
We show that the Drinfeld coproduct can be used to construct non-trivial examples of modules of finite 
type and describe various properties of such modules. 

\medskip

The model example for the analytic Bethe ansatz is the work  \cite{Ba}, where the  $U_q\widehat{\mathfrak{sl}}_2$ 
transfer matrix was studied in relation to the 6-vertex lattice model. It is based on the existence of the operators 
$T(a)$ and $Q(a)$, acting in any $U_q\widehat{\mathfrak{sl}}_2$ module $W$, which satisfy the famous Baxter's $TQ$ relation:
\begin{align}\label{TQ}
T(a)Q(a)=A(a)Q(aq^{-2})+D(a)Q(aq^2)\,.
\end{align}
Here $A(a)$ and $D(a)$ are explicit scalar 
functions depending on $W$. Since $Q(a)$ is known to be polynomial and $T(a)$ regular, one obtains the equations for the 
zeroes of $Q(a)$. These are the Bethe ansatz equations.

The $T$ operator was long known to be the suitably normalized transfer matrix corresponding to the $2$-dimensional 
irreducible evaluation $U_q\widehat{\mathfrak{sl}}_2$ module $V(a)$, $T(a)=\bar T_{V(a)}$.  
Recently, the Baxter's operator $Q(a)$ was identified with a suitably normalized transfer matrix as well. 
This transfer matrix turns out to be related to the positive fundamental module 
(also called the $q$ oscillator representation \cite{BLZ2}).  
Then the relation \eqref{TQ} 
is interpreted as an equation 
in the Grothendieck ring of $U_q\bo$ modules (see e.g. \cite{JMS}):
\begin{align}\label{2 term sl2}
[V_{aq^{-1}}][M_{a}^+]=[M^+_{aq^{-2}}]+[M^+_{aq^2}]\,.
\end{align}
The functions $A(a)$ and $D(a)$ 
appear from normalization (depending on $W$)
of the transfer matrices $T_{V_a}$ and $T_{M^+_{a}}$.

\medskip

For the general quantum algebra, one has the operators $Q_i(a)$, $i\in I$, which are normalized transfer matrices 
related to representations $M^+_{i,a}$. They are shown to be polynomial in \cite{FH}. But there are no 2-dimensional 
modules to complete the argument. In this article
we find an appropriate substitute in the category $\Fin$.

More precisely, we construct modules $N_{i,a}^+$ 
($i\in I$, $a\in\C^\times$) 
which are infinite-dimensional but 2-finite. 
It means that, 
similarly to 
a 2-dimensional $U_q\widehat{\mathfrak{sl}}_2$ module, the algebra  generated $h_{j,r}$, $r>0$, $j\in I$, has only 
two different eigenvalues in $N_{i,a}$. Then we show that
the modules $N_{i,a}^+$ and $M_{i,a}^+$ satisfy 
a two term relation similar to \eqref{2 term sl2}:
\begin{align}\label{NM}
[N_{i,a}^+][M^+_{i,a}]=
\prod\limits_{j: C_{j,i}\neq 0} [M^+_{j,aq_j^{-C_{j,i}}}]
+\prod\limits_{j: C_{j,i}\neq 0} [M^+_{j,aq_j^{C_{j,i}}}]
\,,
\end{align}
where $(C_{i,j})_{i,j\in I}$ is the Cartan matrix. 

For example, in the case of $U_q\widehat{\mathfrak{sl}}_3$, 
consider two irreducible $U_q\bo$ modules with highest $\ell$-weights 
\begin{align*}
\Bigl(q\frac{1-aq^{-2}z}{1-az}\ ,\ 1\Bigr), \quad {\rm{and}} \quad 
\Bigl(q\frac{1-aq^{-2}z}{1-az}\ ,\ (aq)^{-1/2}(1-aqz)\Bigr),
\end{align*}
respectively. The first one is the three dimensional evaluation module. It is a restriction of a $U_q\widehat{\mathfrak{sl}}_3$ module.
The second module is not a restriction of any $U_q\widehat{\mathfrak{sl}}_3$ module. It is an infinite dimensional but $2$-finite module. 
We denote it $N_{1,a}^+$. Then equation \eqref{NM} becomes:
\begin{align*}
[N^+_{1,a}][M_{1,a}^+]=[M^+_{1,aq^{-2}}][M^+_{2,aq}]
+[M^+_{1,aq^2}][M^+_{2,aq^{-1}}]\,.
\end{align*}

The two term relation \eqref{NM} immediately leads to the relation for 
transfer matrices \eqref{TQ-rel} similar to \eqref{TQ}.
We show that transfer matrices related to $U_q\bo$ 
modules $M_{i,a}^+$ and $N_{i,a}^+$ are polynomial 
on each vector after suitable explicit 
normalization (see Proposition \ref{Q pol} and Proposition \ref{T pol}).

Ultimately, it gives a 
proof of the Bethe ansatz equations \eqref{BAE}.

In particular, it implies that the spectrum of a transfer matrix 
related to 
$U_q\bo$ module of category $\cO_\bo$
can be explicitly described in terms of solutions of the Bethe ansatz equation, see Theorem \ref{T-eigv}.
For transfer matrices related to the finite-dimensional modules it was conjectured in \cite{FR}.

The same method works also for quantum toroidal algebras. The simplest case 
of type $\gl_1$ has been studied in our paper \cite{FJMM2}.
\medskip

The paper is constructed as follows. 
We set up the notation, and remind 
the standard facts about the algebras in Section \ref{sec:prelim}, 
and their representations in Section \ref{catO}. 
In Section \ref{q-char sec} we discuss the $q$-characters in the context of Borel subalgebras. 
Then in Section \ref{MV sub} we prove a number of facts about $U_q\bo$ modules which we use later. 
The main results of Section \ref{MV sub} are Lemma \ref{submodules} and Proposition \ref{prop:gradingM}. 
In Section \ref{finite sec} we introduce and study the category of finite type modules. 
In Section \ref{transfer sec} we define the XXZ Hamiltonians, show polynomiality of transfer matrices
and deduce the Bethe ansatz equations.

While preparing this paper, there appeared a work \cite{FH2}
where closely related functional relations for $Q$ operators are derived. 
However the polynomial property which is 
essential for deriving Bethe ansatz equations
is not discussed there. 

\section{Preliminaries}\label{sec:prelim}
In this section we collect basic definitions concerning
quantum loop algebras and their Borel subalgebras.
We follow closely the notation of \cite{FH}. 

\subsection{Notation}\label{notation}

Let $C=(C_{i,j})_{0\leq i,j\leq n}$ be an indecomposable Cartan matrix 
of non-twisted affine type. 
We denote by $\g$ the Kac-Moody Lie algebra associated with $C$. 
Set $I=\{1,\ldots, n\}$, and 
denote by $\gb$  the finite-dimensional 
simple Lie algebra associated with the Cartan matrix $(C_{i,j})_{i,j\in I}$. 
Let $\{\alpha_i\}_{i\in I}$, $\{\alpha_i^\vee\}_{i\in I}$,
$\{\omega_i\}_{i\in I}$, $\{\omega_i^\vee\}_{i\in I}$  
be the simple roots, the simple coroots, the fundamental weights 
and the fundamental coweights of $\gb$, respectively.
We set $Q=\oplus_{i\in I}\Z\,\alpha_i$, 
$Q^+=\oplus_{i\in I}\Z_{\ge0}\,\alpha_i$,
$P=\oplus_{i\in I}\Z\,\omega_i$. 
We denote by $\Delta$ the set of all roots
and by $\Delta^+$ the set of all positive roots of $\gb$. 
Let $D=\mathrm{diag}(d_0,\ldots, d_n)$ be the unique diagonal matrix such that 
$B=DC$ is symmetric and that 
 $d_i$'s are relatively prime positive integers.
We denote by $(~,~):P\times P\to\Q$ 
the invariant symmetric bilinear 
form such that $(\alpha_i,\alpha_i)=2d_i$. 
Let $a_0,\dots,a_n$ stand for the Kac label (\cite{Kac}, pp.55-56).
We have $a_0=1$. 

Throughout this paper, we fix  a non-zero complex number 
$q$ which is not a root of unity. 
We fix $\hbar\in\C$ such that $q=e^\hbar$, and 
write $q^\lambda=e^{\lambda\hbar}$ for $\lambda\in\C$. 
We set $q_i=q^{d_i}$, $q_{i,j}=q^{(\alpha_i,\alpha_j)}=q_{j,i}$.
We use the standard symbols for $q$-integers
\begin{align*}
[m]_v=\frac{v^m-v^{-m}}{v-v^{-1}}\,, \quad
[m]_v!=\prod_{k=1}^m [k]_v\,,
 \quad 
\qbin{m}{k}_v
=\frac{[m]_v!}{[k]_v![m-k]_v!}\,. 
\end{align*}

\subsection{Quantum loop algebra}\label{debut}
The quantum loop algebra $U_q\g$ is the $\C$-algebra defined by generators 
$e_i,\ f_i,\ k_i^{\pm1}$ ($0\le i\le n$) 
and the following relations for $0\le i,j\le n$.
\begin{align*}
&k_ik_j=k_jk_i,\quad k_0^{a_0}k_1^{a_1}\cdots k_n^{a_n}=1,\\
&k_ie_jk_i^{-1}=q_{i,j}e_j,\quad k_if_jk_i^{-1}=q_{i,j}^{-1}f_j,
\\
&[e_i,f_j]
=\delta_{i,j}\frac{k_i-k_i^{-1}}{q_i-q_i^{-1}},
\\
&\sum_{k=0}^{\ell_{i,j}}(-1)^k x_i^{(\ell_{i,j}-k)}x_j x_i^{(k)}=0\quad (i\neq j,\ x=e,f).
\end{align*}
In the last relation, we set $\ell_{i,j}=1-C_{i,j}$ and 
$x_i^{(k)}=x_i^k/[k]_{q_i}!$ for $x_i=e_i,f_i$.

The algebra $U_q\g$ can also be presented in terms
of the Drinfeld generators \cite{Dr,Be}
\begin{align*}
x_{i,m}^{\pm}\ (i\in I, m\in\Z), 
\quad 
h_{i,r} \ (i\in I, r\in \Z\backslash\{0\}),
\quad k^{\pm1}_i\ (i\in I)\,.
\end{align*}
We shall use the generating series 
\begin{align*}
&x^\pm_i(z)=\sum_{m\in\Z}x^{\pm}_{i,m}z^{m}\,,\\
&\phi^\pm_i(z)=
\sum_{m\in\Z}\phi^{\pm}_{i,m}z^{m}=
k_i^{\pm1}\exp\Bigl(
\pm(q_i-q_i^{-1})\sum_{\pm r>0}h_{i, r}z^{r}\Bigr)\,.
\end{align*}
Then for all $i,j\in I$ and $r,s\in\Z\backslash\{0\}$ we have
\begin{align}
&\phi_i^\epsilon(z)\phi_j^{\epsilon'}(w)
=\phi_j^{\epsilon'}(w)\phi_i^\epsilon(z)
\quad (\epsilon,\epsilon'\in\{+,-\})\,,
\label{phi-phi}\\
&(q_{i,j}^{\pm 1}z-w)
\phi^\epsilon_i(z)x_j^\pm(w)
=(z-q_{i,j}^{\pm 1}w)
x_j^\pm(w)\phi^\epsilon_i(z)
\quad (\epsilon\in\{+,-\}),
\label{phi-x}\\
&
[x_i^+(z),x_j^-(w)]
=\delta_{i,j}\delta(z/w)
\frac{\phi^+_i(z)-\phi^-_i(z)}{q_i-q_i^{-1}}\,,
\label{x+x-}\\
&
\bigl(q_{i,j}^{\pm 1} z-w\bigr)
x_i^{\pm}(z) x_j^{\pm}(w)
=
\bigl(z-q_{i,j}^{\pm 1} w\bigr)
x_j^{\pm}(w)x_i^{\pm}(z) \,,
\label{x-x}\\
&
\mathop{\mathrm{Sym}}_{z_1,\dots,z_{\ell_{i,j}}}
\sum_{k=0}^{\ell_{i,j}}
(-1)^k
\qbin{\ell_{i,j}}{k}_{q_i}
x_i^{\pm}(z_1)\cdots 
x_i^{\pm}(z_k)
x_j^{\pm}(w)
x_i^{\pm}(z_{k+1})\cdots x_i^{\pm}(z_{\ell_{i,j}})=0\,\quad (i\neq j). 
\label{x-Serre}
\end{align}
In the last line we set 
\begin{align*}
\mathop{\mathrm{Sym}}_{z_1,\dots,z_\ell}f(z_1,\dots,z_\ell)=
\frac{1}{\ell!}
\sum_{\sigma\in S_\ell} f\bigl(z_{\sigma(1)},\dots,z_{\sigma(\ell)}\bigr)\,.
\end{align*}

We have in particular
\begin{align}\label{hx}
&k_ix^{\pm}_j(z)k_i^{-1}=q_{i,j}^{\pm 1}x^\pm_j(z)\,,\notag
\\
&[h_{i,r}, x^{\pm}_j(z)]= \pm\frac{[r C_{i,j}]_{q_i}}{r} z^{-r} x^\pm_j(z)\,.
\end{align}

Let $U^\pm_q\g$ 
be the subalgebra of $U_q\g$ generated by the 
$x_{i,m}^\pm$ ($i\in I, m\in\Z$), 
and let $U^0_q\g$ be the one generated 
by the $k_i^{\pm 1}$, $h_{i,r}$ ($i\in I,r\in\Z\backslash\{0\}$). 
We have a triangular decomposition
\begin{align}
U_q\g\simeq 
U_q^-\g\otimes U^0_q\g \otimes U^+_q\g\,.
\label{triangular-g} 
\end{align}
We denote by $\tb$ the subalgebra of $U_q\g$ 
generated by the $k_i^{\pm1}$ $(i\in I)$.

The algebra $U_q\g$ has a $Q\times \Z$-grading given by
\begin{align*}
 &\deg e_i=(\alpha_i,0),\quad \deg f_i=(-\alpha_i,0),\quad 
\deg k_i=(0,0)\,\quad (i\in I),\\
&\deg e_0=(\accentset{\circ}{\alpha}_0,1),\quad
\deg f_0=(-\accentset{\circ}{\alpha}_0,-1),
\quad \deg k_0=(0,0)\,,
\end{align*}
where $\accentset{\circ}{\alpha}_0=-\sum_{i\in I}a_i\alpha_i$. 
We have $\deg x^\pm_{i,m}=(\pm\alpha_i,m)$ and 
$\deg h_{i,r}=(0,r)$
for $i\in I$, $m\in\Z$, 
$r\in \Z\backslash\{0\}$.
If $\deg x=(\beta,m)$ then 
we say that $x$ has 
{\it weight} $\beta$ 
and 
{\it homogeneous degree} $m$,  
and write $\beta=\wt x$, 
$m=\hdeg x$.
We denote by $(U_q\g)_{\beta}$ the graded component of $U_q\g$ of 
weight $\beta$.   

\subsection{Hopf algebra structure}

The algebra $U_q\g$ has a Hopf algebra structure
\begin{align*}
&\Delta(e_i)=e_i\otimes 1+k_i\otimes e_i,\quad
\Delta(f_i)=f_i\otimes k_i^{-1}+1\otimes f_i,
\quad
\Delta(k_i)=k_i\otimes k_i\,,
\\
&\varepsilon(e_i)=0,\quad \varepsilon(f_i)=0,\quad 
\varepsilon(k_i)=1\,,\\
&
S(e_i) = -k_i^{-1} e_i ,\quad S(f_i) = -f_i k_i,
\quad
S(k_i)=k_i^{-1}\,,
\end{align*}
where $i=0,\dots,n$. 

The following Proposition gives partial information about 
the coproduct of the Drinfeld generators. 
\begin{prop}\cite{Da}\label{apco} 
For $i\in I$ and $r > 0$, we have
\begin{equation}\label{h}
\Delta\left(\phi_{i,\pm r}^\pm \right) \in 
\sum_{0\leq l\leq r} \phi_{i,\pm l}^\pm \otimes \phi_{i,\pm (r - l)}^\pm
+\sum_{\beta\in Q^+\backslash\{0\}} \bigl(U_q\g\bigr)_{-\beta} \otimes \bigl(U_q\g\bigr)_{\beta}\,.
\end{equation}
\end{prop}

\subsection{Borel algebras}\label{boralg}

We define two subalgebras of $U_q\g$
as follows.
\begin{align*}
& U_q\bo=\langle e_i\ (0\le i\le n),\quad k_i^{\pm 1} (i\in I)\rangle,
\\
& U_q\bbo=\langle f_i\ (0\le i\le n),\quad k_i^{\pm 1} (i\in I)\rangle\,.
\end{align*}
We call  $U_q\bo$ the {\it positive Borel algebra} (or simply the Borel algebra), 
and $ U_q\bbo$ the {\it negative Borel algebra}.
Algebras $U_q\bo$, $U_q\bbo$  are both Hopf subalgebras of $U_q\g$.   

Set $U^\pm_q\bo = U^\pm_q\g\cap U_q\bo$ and 
$U^0_q\bo = U^0_q\g\cap U_q\bo$. 
From the result of \cite{Be} 
it follows that we have the triangular decomposition
\begin{equation}\label{triandecomp}
U_q\bo\simeq U^-_q\bo\otimes 
U^0_q\bo \otimes U^+_q\bo.
\end{equation}
In terms of the Drinfeld generators we have
\begin{align*}
&U^+_q\bo = \langle x_{i,m}^+\ (i\in I, m\geq 0)\rangle\,,
\quad
U^0_q\bo =\langle h_{i,r},k_i^{\pm1} \ (i\in I, r>0)\rangle\,.
\end{align*} 
We have also 
$U^-_q\bo \supset \langle x_{i,r}^- \ (i\in I, r> 0)\rangle$
but the inclusion is proper except in the case $\gb=\slt$.

In \cite{Be,Da}, certain root vectors 
$e_\beta\in U_q\bo$, $f_\beta\in U_q\bbo$ are introduced with each positive real root $\beta$ of $\g$. 
In Appendix we give a review of their definition and basic facts. 
The subalgebras $U^\pm_q\bo$ are generated by root vectors, 
\begin{align*}
&U^+_q\bo=\langle e_{m\delta+\alpha}\mid m\ge0,\ \alpha\in\Delta^+ \rangle\,,
\\
&U^-_q\bo=\langle k_{\alpha}
e_{m\delta-\alpha}\mid m>0,\ \alpha\in\Delta^+ \rangle
\,,
\end{align*}
where we set $k_\alpha=\prod_{i\in I}k_i^{b_i}$ 
for $\alpha=\sum_{i\in I}b_i\alpha_i$. 

The root vectors have a convexity property \eqref{convex}
with respect to a total ordering 
\begin{align}
\beta_0\prec\beta_{-1}\prec\beta_{-2}\prec\cdots\prec
2\delta\prec\delta\prec \cdots
\prec\beta_3\prec\beta_2\prec\beta_1\,, 
\label{total-order}
\end{align}
where 
$\{\beta_r\mid r\le0\}=\{k\delta+\alpha\mid k\ge 0,\ \alpha\in \Delta^+\}$,
$\{\beta_r\mid r\ge1\}=\{k\delta-\alpha\mid k>0,\ \alpha\in \Delta^+\}$; 
see \eqref{real-rootv1}, \eqref{real-rootv2}, \eqref{root-order}.

\subsection{Universal $R$ matrix}
It is well known that $U_q\g$ is equipped with the universal $R$ matrix
$\cR\in U_q\bo\widehat{\otimes}U_q\bbo$, which satisfies 
\begin{align*}
&\cR\, \Delta(x)=\Delta^{op}(x) \, \cR\,\quad(x\in U_q\g),
\\
&\bigl(\Delta\otimes\id\bigr)\cR=\cR_{1,3}\cR_{2,3}\,,\\
&\bigl(\id\otimes\Delta\bigr)\cR=\cR_{1,3}\cR_{1,2}\,.
\end{align*}
Here $\Delta^{op}=\sigma\circ\Delta$, $\sigma(a\otimes b)=b\otimes a$, and 
$\cR_{1,2}=\cR\otimes 1$, etc..
It has the product form
\begin{align}
&\cR=\cR_+\cR_0\cR_- q^{-t_\infty} \,, 
\label{univR}
\end{align}
where each factor is given in terms of root vectors as follows. 
\begin{align}
&\cR_+=\prod_{r\le0}
\exp_{q_{\beta_r}}\left(-(q_{\beta_r}-q^{-1}_{\beta_r})
e_{\beta_r}\otimes f_{\beta_r}\right)\,,
\label{R+}\\
&\cR_0=
\exp\Bigl(-\sum_{r>0\atop i,j\in I}\frac{r\widetilde{B}_{i,j}(q^r)}{q^r-q^{-r}}
(q_i-q_i^{-1})(q_j-q_j^{-1})
\, h_{i,r}\otimes h_{j,-r}\Bigr)\,,
\label{R0}\\
&\cR_-=\prod_{r\ge1}
\exp_{q_{\beta_r}}\left(-(q_{\beta_r}-q^{-1}_{\beta_r})
e_{\beta_r}\otimes f_{\beta_r}\right)\,.
\label{R-}
\end{align}
In these formulas, we set $q_\beta=q^{(\beta,\beta)/2}$, 
$\exp_q(x)=\sum_{n=0}^\infty x^n \prod_{j=1}^n (q^2-1)/(q^{2j}-1)$. 
The matrix $\bigl(\widetilde{B}_{i,j}(q)\bigr)_{i,j\in I}$ is the inverse matrix of 
$\bigl([d_iC_{i,j}]_q\bigr)_{i,j\in I}$.  
In \eqref{R+} and  \eqref{R-}, the product is ordered 
from left to right in the decreasing order of $r$ as in  \eqref{total-order}.
The element $t_\infty$ is formally given by
$\sum_{i,j\in I}d_i(C^{-1})_{i,j}h_{i,0}\otimes h_{j,0}$, 
where we set $k_i=q_i^{h_{i,0}}$. 
The expression $q^{-t_\infty}$ has a well defined meaning 
on a tensor product of weight modules. 

\subsection{Drinfeld coproduct}

Along with the `standard' coproduct $\Delta$ introduced above,  
we also make use of the so-called Drinfeld coproduct $\Delta_D$
given by 
\begin{align}\label{Drin coprod}
&\Delta_D\bigl(x^+_i(z)\bigr)=x^+_i(z)\otimes 1
+\phi^-_i(z)\otimes x^+_i(z)\,,
\notag
\\
&\Delta_D\bigl(x^-_i(z)\bigr)=
x^-_i(z)\otimes \phi^+_i(z)
+ 1\otimes x^-_i(z)\,,
\\
&\Delta_D\bigl(\phi^\pm_i(z)\bigr)=\phi^\pm_i(z)\otimes \phi^\pm_i(z)\,. \notag
\end{align}
Since the terms in the right hand side involve infinite sums of generators,  
$\Delta_D$ is not a coproduct in the usual sense. Nevertheless, 
under certain circumstances it can be used to define a module structure
on tensor products of representations, see Section \ref{MV subsec}. 

The Drinfeld coproduct $\Delta_D$ and the standard coproduct $\Delta$
are related to each other through the factor \eqref{R+} of the universal $R$ matrix:
\begin{prop}\cite{EKP}\label{prop:twist}
For any $x\in U_q\g$ we have
\begin{align*}
\Delta_D(x)=\sigma(\cR_+)^{-1}\cdot \Delta(x)\cdot\sigma(\cR_+)\,.
\end{align*} 
\end{prop}

The Borel subalgebra 
$U_q\bo$ is {\it not} a Hopf subalgebra of $U_q\g$
with respect to $\Delta_D$. 
Proposition \ref{prop:twist} implies that  
it is rather a coideal, i.e.,
\begin{align}
&\Delta_D\bigl(U_q \bo\bigr)\, \subset \,
U_q\g\,\widehat{\otimes}\, U_q\bo\,.
\label{coideal-b}
\end{align}

\section{Representations of $U_q\g$ and $U_q\bo$}\label{catO}

We review known facts about representations of $U_q\g$ and $U_q\bo$.

\subsection{Weights and $\ell$-weights}\label{sec:wts}

We begin by introducing some terminology. 

First let $\tb^*=\bigl(\C^\times\bigr)^I$. 
For $\lambda\in P$ we define elements
$\sfq^\lambda=\bigl(q_j^{(\lambda,\alpha_j^\vee)}\bigr)_{j\in I}\in \tb^*$.
For a $U_q\g$ module $V$ and 
$\mu=\bigl(\mu_i\bigr)_{i\in I}\in \tb^*$, we set 
\begin{align*}
V_{\mu}=\{v\in V \mid k_i\, v = \mu_i v\quad (i\in I)\}\,.
\end{align*}
We have then 
$\phi_{i,r}^\pm (V_\mu)\subset V_\mu$
and 
$x_{i,r}^\pm (V_{\mu}) \subset V_{\mu\, \sfq^{\pm \alpha_i}}$
for all $i\in I$ and $r\in\Z$.  
We say that $\mu\in \tb^*$ is a weight of $V$ if $V_\mu\neq0$.
The set of weights of $V$ is denoted by $\wt V$.
We say that $V$ is $\tb$-diagonalizable if 
$V=\underset{\mu\in \wt V}{\bigoplus}V_{\mu}$. 

Next let $\tb^*_{\ell,\g}$ denote the set of all pairs 
$\Psib = (\Psib^+,\Psib^-)$ consisting of $I$-tuples of formal series
\begin{align*}
&\Psib^\pm=\bigl(\Psi^\pm_i(z)\bigr)_{i\in I},
\quad
\Psi^\pm_i(z) = \sum_{\pm m\geq 0}\Psi^\pm_{i,m} z^{ m}
\in \C[[z^{\pm 1}]]\,,
\end{align*}
such that $\Psi^+_{i,0}\Psi^-_{i,0}=1$ for all $i\in I$. 
The sets $\tb^*$ and  $\tb^*_{\ell,\g}$ are both 
abelian groups by pointwise multiplication, 
and we have a group homomorphism $\varpi:\tb^*_{\ell,\g}\rightarrow \tb^*$ 
given by $\varpi(\Psib)=\bigl(\Psi^+_{i,0}\bigr)_{i\in I}$.

For a $\tb$-diagonalizable $U_q\g$ module
$V=\underset{\mu\in \wt V}{\bigoplus}V_{\mu}$ 
and $\Psib\in\tb_{\ell,\g}^*$, we set
\begin{align*}
V_{\Psibs} &=
\{v\in V_{\varpi(\Psibs)}
\mid \text{there exists a $p\ge0$ such that}
\\
&\text{$(\phi_{i,m}^\epsilon - \Psi^\epsilon_{i,m})^pv = 0$ 
($i\in I, \epsilon m\geq 0, \epsilon\in\{+,-\}$)}\}\,.
\nn
\end{align*}
If $V_{\Psibs} \neq0$, then we call it 
 $\ell$-weight space of $V$ of $\ell$-weight $\Psib$.

We say that a $U_q\g$ module $V$ is a highest $\ell$-weight module of 
highest $\ell$-weight $\Psib\in\tb^*_{\ell,\g}$ 
if it is generated by a non-zero vector $v\in V$
such that
\begin{align*}
U^+_q\g \cdot v=\C v\,,\quad \phi^\epsilon_i(z)v=\Psi^\epsilon_i(z)v
\quad (i\in I,\ \epsilon\in\{+,-\}). 
\end{align*}
If it is the case, we say $v$ is a highest $\ell$-weight vector of $V$. 
For each $\Psib\in\tb^*_{\ell,\g}$, there exists a unique 
simple highest $\ell$-weight module of highest $\ell$-weight $\Psib$. 
We denote it by $L(\Psib)$. Owing to the triangular decomposition
\eqref{triangular-g}, $\wt L(\Psib)$ is contained in the set
\begin{align*}
D(\mu)=\{\mu\, {\sf q}^{-\beta}\mid \beta\in Q^+\}
\end{align*}
where $\mu=\varpi(\Psib)$. 

For $U_q\bo$ modules, weight and weight space
 are defined in the same way as above. 
The notion of $\ell$-weight is defined similarly, using a single 
$I$-tuple of formal series
\begin{align*}
&\Psib=(\Psi_i(z))_{i\in I},
\quad
\Psi_i(z) = \sum_{m\ge0} \Psi_{i, m} z^{m}\in \C[[z]]\,
\end{align*}
satisfying $\Psi_{i,0}\neq 0$. 
We denote by $\tb^*_{\ell,\bo}$ the set of all such $\Psib$'s. 
A highest $\ell$-weight module $V$ of $U_q\bo$
is defined by the conditions that 
$V=U_q\bo\cdot v$, $v\neq 0$, and 
\begin{align*}
U^+_q\bo\cdot v=\C v\,,\quad \phi^+_i(z)v=\Psi_i(z)v
\quad (i\in I). 
\end{align*}
The unique simple highest $\ell$-weight module of 
highest $\ell$-weight $\Psib\in \tb^*_{\ell,\bo}$ 
is denote by $L(\Psib)$.

\subsection{Category $\cO_{\g}$}\label{sec:Og}

We consider a full subcategory $\cO_{\g}$ 
of the category of all $U_q\g$ modules. 
By definition, a $U_q\g$ module $V$ is an object of
category $\cO_{\g}$ if the following conditions (i)--(iii) hold:
\begin{enumerate}
\item $V$ is $\tb$-diagonalizable,
\item $\dim V_\mu<\infty$ for all $\mu$, 
\item There exist $\mu_1,\cdots,\mu_N\in \tb^*$ such that
$\wt V\subset D(\mu_1)\cup\cdots\cup D(\mu_N)$. 
\end{enumerate}
Category $\cO_{\g}$ is a monoidal category.

Simple objects of $\cO_{\g}$ are classified by 
highest $\ell$-weights as explained below. 
We say that $\Psib\in\tb^*_{\ell,\g}$ is rational 
if there exists 
an $I$-tuple of rational functions $\{f_i(z)\}_{i\in I}$   
such that $f_i(z)$ is regular at $z=0,\infty$, satisfies
$f_i(0)f_i(\infty)=1$, and that  
$\Psi_i^\pm(z)$ are expansions of $f_i(z)$ at $z^{\pm 1}=0$
for all $i\in I$.  
If it is the case, we do not make distinction between the formal series
$\Psi^\pm_i(z)$ and the rational function $f_i(z)$. 
We denote by $\mfr_{\g}$
the set of rational elements of $\tb^*_{\ell,\g}$. 

\begin{thm}\label{thm:simple-g}\cite{MY} 
Let $\Psib\in \tb^*_{\ell,\g}$. Then  
all weight spaces of $L(\Psib)$ are finite dimensional 
if and only if $\Psib$ is rational. 
The map $\Psib\mapsto L(\Psib)$ gives a bijection
between the set of all rational 
$\Psib\in\mfr_{\g}$
 and isomorphism classes of 
simple objects of $\cO_{\g}$. 
\end{thm}
We quote also a classical result concerning finite dimensional modules. 
A $\tb$-diagonalizable $U_q\g$ module is said to be 
of type $1$ if all weights
are of the form $\bigl(q_i^{(\Lambda,\alpha_i^\vee)}\bigr)_{i\in I}$
with some $\Lambda\in P$. 
\begin{thm}\cite{CP}\label{Ch-Pr}
A $U_q\g$ module $V$ of type $1$ is finite dimensional if and only if 
$V= L(\Psib)$ with some 
$\Psib\in\mfr_{\g}$
which has the form 
\begin{align*}
\Psi_i^\pm(z)=q_i^{\deg P_i}\frac{P_i(q^{-1}_i z)} {P_i(q_i z)} 
\quad (i\in I)\,,
\end{align*}
where $P_i(z)$ is a monic polynomial satisfying $P_i(0)=1$.
\end{thm}

As in the case of Kac-Moody Lie algebras, 
\cite{Kac}, Section 9.6, the notion of the multiplicity $[V:L(\Psib)]$ 
of a simple object $L(\Psib)$, $\Psib\in \mfr_{\g}$, in $V\in \Ob\mathcal{O}_\g$
is well-defined. 
Following \cite{HL}, we define the Grothendieck ring $\Rep\, U_q\g$ of category $\cO_\g$ as follows.
By definition,  $\Rep\, U_q\g$ is an additive group of maps $c:\mfr_{\g}\to \Z$ 
such that $\supp(c)$ is contained in a finite union of the $D(\mu)$'s, 
and that  $\supp(c)\cap\varpi^{-1}(\omega)$ is a finite set for any $\omega\in\tb^*$.
Here we write $\supp(c)=\{\Psib\in\mfr_{\g}\mid c(\Psib)\neq0\}$. 
The ring structure is defined by
 \[
(cc')(\Psib'')=\sum_{\Psibs,\Psibs'\in\mfr_{\g}\atop \Psibs\Psibs'=\Psibs''}c(\Psib)c'(\Psib')[L(\Psib)\otimes L(\Psib'):L(\Psib'')].
\]
We write an element $c$ of  $\Rep\, U_q\g$ also as a formal sum
$\sum_{\Psibs\in \mfr_{\g}} c(\Psib)[L(\Psib)]$
where $[L(\Psib)]$ means the map $\Psib'\mapsto \delta_{\Psibs,\Psibs'}$.
For $V\in\Ob \cO_\g$ we set $[V]=\sum_{\Psibs\in\mfr_{\g}}[V:L(\Psib)][L(\Psib)]\in\Rep\,U_q\g$. 
We have $[V_1\otimes V_2]=[V_1][V_2]$ and $[V_1\oplus V_2]=[V_1]+[V_2]$.
In addition, for any short exact sequence of $U_q\g$ modules $0\to V_1\to V_2\to V_3\to 0$ we have the relation
$[V_2]=[V_1]+[V_3]$ in $\Rep\,U_q\g$.

\subsection{Category $\cO_{\bo}$}\label{sec:Ob}

We define category $\cO_{\bo}$ of $U_q\bo$ modules 
by the same conditions (i)--(iii) as in the previous subsection, 
replacing $U_q\g$ modules with $U_q\bo$ modules. 
We denote $\Rep\,U_q\bo$ the corresponding Grothendieck ring. 
We say that $\Psib\in\tb^*_{\ell,\bo}$ is rational 
if there exists 
an $I$-tuple of rational functions $\{f_i(z)\}_{i\in I}$   
such that $f_i(z)$ is regular and non-zero at $z=0$, 
and that 
$\Psi_i(z)$ is an expansion of $f_i(z)$ at $z=0$  
for all $i\in I$. 
We denote by $\mfr_{\bo}$
the set of rational elements of $\tb^*_{\ell,\bo}$.

The following result is a counterpart to Theorem \ref{thm:simple-g}. 
\begin{thm}\cite{HJ} 
Let $\Psib\in \tb^*_{\ell,\bo}$. Then  
all weight spaces of $L(\Psib)$ are finite dimensional 
if and only if $\Psib$ is rational. 
The map $\Psib\mapsto L(\Psib)$ gives a bijection
between the set of all rational elements 
$\Psib\in\mfr_{\bo}$
 and isomorphism classes of 
simple objects of $\cO_{\bo}$. 
\end{thm}

We have an inclusion of spaces of rational functions
$\mfr_\g\hookrightarrow \mfr_{\bo}$. 
\begin{lem}\label{lem:simple-to-simple}
The restriction functor $\mathrm{Res}:\cO_{\g}\rightarrow \cO_{\bo}$ 
sends simple objects to simple objects. 
\end{lem}
\begin{proof}
The proof is similar to those given in Proposition 3.5 of \cite{HJ} and Lemma 3.3 of \cite{FJMM2}. 

Let $V$ be an object in $\cO_\g$.
We fix $\mu\in \tb^*$ and $i\in I$. 

Consider the set of operators 
$x^+_{i,m}:V_{\mu}\to V_{\mu\sfq^{\alpha_i}}$ ($m\ge 0$).
Since $\Hom_\C(V_\mu,V_{\mu\sfq^{\alpha_i}})$ is finite dimensional, 
we have a linear relation  $\sum_{j=a}^b c_j x^+_{i,j}|_{V_\mu}=0$ 
where $c_j\in\C$ ($0<a\le j\le b$) and $c_ac_b\neq0$. 
Taking commutators with 
$h_{i,\pm1}$ we obtain $\sum_{j=a}^b c_j x^+_{i,j+r}|_{V_\mu}=0$ 
for all $r\in \Z$. 
It follows that operators 
$x^+_{i,k}|_{V_\mu}$ ($k\in\Z$) 
belong to the linear span of $\{x^+_{i,m}|_{V_\mu}\}_{m\ge0}$.
By the same argument, operators $x^-_{i,k}|_{V_\mu}$ ($k\in\Z$) 
belong to the linear span of $\{x^-_{i,m}|_{V_\mu}\}_{m> 0}$. 

It is clear now that any singular vector of $V$ with respect to $U_q\bo$
is also singular with respect to $U_q\g$. It is also clear that
if $V$ is cyclic with respect to $U_q\g$ then it is cyclic with respect to 
$U_q\bo$. The assertion of Lemma follows from these.
\end{proof}

There are objects of $\cO_{\bo}$ which cannot be obtained 
by restricting simple objects of $\cO_{\g}$. 

For example, the Borel algebra $U_q\bo$ has a large family of one-dimensional modules 
labeled by $\tb^*$. 
These are the modules $L(K)$ whose highest $\ell$-weights
are constants, $\Psi_i(z)=K_i\in \C^\times$. 
The only one-dimensional $U_q\bo$ module 
which is obtained as a restriction corresponds 
to the choice $K_i^2=1$, $i\in I$.

Nontrivial examples are the modules $M^\pm_{i,a}$
($i\in I$, $a\in \C^{\times}$) defined as follows. 
\begin{align*}
M^\pm_{i,a}=L(\Psib) 
\quad\text{ where }\quad
\Psi_j(z)=\begin{cases}
	   a^{\mp 1/2}(1-a z)^{\pm1} & \text{for $j=i$},\\
           1 & \text{for $j\neq i$}.
	  \end{cases}
\end{align*} 

We call $M^+_{i,a}$ {\it positive fundamental module},
and $M^-_{i,a}$ {\it negative fundamental module}.
We note that in \cite{FH} the modules 
$M^\pm_{i,a}\otimes L\bigl((a^{\pm \delta_{ij}/2})_{j\in I}\bigr)$ 
are called `prefundamental'.

\subsection{Dual category $\cO^\vee_{\bo}$}\label{subsec:dualO}

We say that a $U_q\bo$ module $V$ is 
an object of category $\cO^\vee_{\bo}$ if the following are satisfied.
\begin{enumerate}
\item $V$ is $\tb$-diagonalizable,
\item $\dim V_\mu<\infty$ for all $\mu$, 
\item There exist $\mu_1,\cdots,\mu_N\in \tb^*$ such that
$\wt V\subset D(\mu_1)^{-1}\cup\cdots\cup D(\mu_N)^{-1}$. 
\end{enumerate}

A $U_q\bo$ module $V$ is said to be 
of lowest $\ell$-weight $\Psib\in \tb^*_{\ell,\bo}$ if it is generated by a 
non-zero $v\in V$ satisfying 
\begin{align*}
U^-_q\bo \cdot v=\C v\,,\quad 
\phi_{i}^+(z)v=\Psi_{i}(z)v\quad (i\in I)\,.
\end{align*}
The unique simple $U_q\bo$ 
module of lowest $\ell$-weight $\Psib\in\tb^*_{\ell,\bo}$ is 
denoted $L^\vee(\Psib)$.

Let $V=\oplus_{\mu\in\wt V}V_{\mu}$ be a $U_q\bo$
module with finite dimensional weight spaces $V_{\mu}$. 
Let  $V^*=\oplus_{\mu\in\wt V}V^*_{\mu}$ be the graded dual space.
We introduce a structure of left $U_q\bo$ module on $V^*$ by setting 
\begin{align*}
(x v^*)(v)=v^*\bigl(S^{-1}(x)v\bigr)\,\quad
(v\in V,\ v^*\in V^*,\ x\in U_q\bo).
\end{align*}

\begin{lem}\cite{HJ}\label{dual lem}
A $U_q\bo$ module $V$ is in category $\cO^\vee_{\bo}$ if and only if 
$V^*$ is in category $\cO_{\bo}$.  
We have 
\begin{align*}
 L^\vee(\Psib)=\bigl(L(\Psib^{-1})\bigr)^*\,.
\end{align*} 
\end{lem}

\section{The $q$ characters.}\label{q-char sec}
The $q$-characters encode the $\ell$-weights of 
a representation and provide a useful tool \cite{FR}. 
Following \cite{HJ}, we recall their definition in the context of 
category $\mc O_\bo$ of $U_q\bo$ modules. 
Then we establish a few useful properties. 
With obvious changes, all the statements 
can be similarly proved for 
the category $\mc O^\vee_\bo$ of lowest weight modules as well.

\subsection{The definition of the $q$-characters.}

First we prepare some notation. 

In the group ring $\Z[\tb^*]$,    
we use the letter $y_i^b$ 
to denote the element 
$\bigl(q_i^{b\delta_{i,j}}\bigr)_{j\in I}\in \tb^*$ where $i\in I$, $b\in \C$.  
We have $y_i^{b}y_i^{b'}=y_i^{b+b'}$, 
and $\sfq^{\omega_i}=y_i$, $\sfq^{\alpha_i}=\prod_{j\in I}y_j^{C_{j,i}}$.
We identify  $\Z[\tb^*]$ with the ring $\Z[y_i^b]_{i\in I,b\in\C}$.    

The character of $V\in \Ob \cO_{\bo}$ is the 
generating series of weight multiplicities defined by
\begin{align*}
\chi(V)=\sum_{\mu} \dim V_{\mu} \cdot \mu\,.
\end{align*}
By the condition (iii) for category $\cO_{\bo}$, the right hand side 
belongs to the ring
$\mc X_0=\Z[[\sfq^{-\alpha_i}]][y_i^b]_{i\in I, b\in \C}$
consisting of polynomials in $y_i^b$'s whose coefficients are 
formal power series in the variables $\sfq^{-\alpha_i}$.

For each $i\in I$ and $a\in \C^\times$
we introduce further a new independent variable $X_{i,a}$, and set 
\begin{align*}
Y_{i,a}=\frac{X_{i,aq_i^{-1}}}{X_{i,aq_i}}\,,\quad
A_{i,a}
=\prod_{j\in I}\frac{X_{j, aq_{j,i}^{-1}}}{X_{j,aq_{j,i}}}\,.
\end{align*}
The monomials $Y_{i,a}$ and $A_{i,a}$ are affine analogs of fundamental weights and roots, respectively. 
We have
\begin{align*}
A_{i,a}
=Y_{i,aq_i^{-1}}Y_{i,aq_i}\prod_{j\in I, C_{ji}=-1}Y_{j,a}^{-1}
\prod_{j\in I, C_{ji}=-2}Y_{j,aq^{-1}}^{-1}Y_{j,aq}^{-1}\prod_{j\in I, C_{ji}=-3}Y_{j,aq^{-2}}^{-1}Y_{j,a}^{-1}Y_{j,aq^2}^{-1}\,.
\end{align*}
Note that the $A_{i,a}$'s are algebraically independent.

Define a group isomorphism $m$ between the 
multiplicative group $\mfr_{\bo}$ and the group of monomials in the $X_{i,a}$'s and $y_i^b$'s 
as follows. 
If $\Psib=\bigl(\Psi_i(z)\bigr)_{i\in I}\in\mfr_{\bo}$ has the form 
\begin{align*}
&
\Psi_i(z)=q_i^{b_i}
\prod_{r=1}^{k_i}(a_{i,r}^{-1/2}-a_{i,r}^{1/2}{z})
\prod_{s=1}^{l_i}(b_{i,s}^{-1/2}-{b_{i,s}^{1/2}}{z})^{-1}
\quad (b_i\in\C\,,a_{i,r}\,,b_{i,s}\in \C^\times),
\end{align*}
then we set
\begin{align*}
&m(\bs\Psi)=\prod_{i\in I}
\bigl(y_i^{b_i}
\prod_{r=1}^{k_i}X_{i,a_{i,r}}\prod_{s=1}^{l_i}X_{i,b_{i,s}}^{-1}
\bigr)\,.
\end{align*}
We use monomials to label $\ell$-weights. 
Namely, if $\bm=m(\bs\Psi)$, then 
we write the irreducible highest $\ell$-weight module $L(\Psib)$ as
$L(\bm)$, and the $\ell$-weight space $V_{\bs\Psi}$ as $V_\bm$. 
For example, we have $L(X_{i,a}^{\pm1})=M^{\pm}_{i,a}$.

The $q$-character of $V\in \Ob \cO_{\bo}$ is the 
generating series of $\ell$-weight multiplicities defined by
\begin{align*}
\chi_q(V)=
\sum_{\bm} \dim V_{\bm}  \cdot \bm\,.
\end{align*}
We say that a monomial $\bm$ is in $\chi_q(V)$ if $\dim V_\bm\neq 0$.
It is known that $\chi_q(V)$ belongs to the ring 
$\mc X=\Z[[A_{i,a}^{-1}]][X_{i,c}^{\pm1}, y_i^b]_{i\in I, a,c\in\C^\times,b\in\C}$
consisting of polynomials in $X_{i,c}^{\pm1}$'s and $y_i^b$'s, 
whose coefficients are formal power series 
in the variables $A_{i,a}^{-1}$.

There is a ring homomorphism 
$\varpi : \mc X\to \mc X_0$ 
given by 
$\varpi(X_{i,a})=y_i^{-\log a /(2\log q_i)}$ ($i\in I,\ a\in \C^\times$),
$\varpi(y_i^b)=y_i^b$ ($b\in\C$). 
We have $\varpi(Y_{i,a})=y_i$, $\varpi(A_{i,a})=\sfq^{\al_i}$,
 and $\varpi(\chi_q(V))=\chi(V)$ for any $V\in \Ob \cO_\bo$.

The following was stated in \cite{FH}.
\begin{prop}\label{inj} The $q$-character map
\begin{align*}
\chi_q:\ \Rep\,U_q\bo \to \mc X,
\qquad V\mapsto \chi_q(V),
\end{align*}
is an injective ring homomorphism.
\end{prop}
\begin{proof} The map $\chi_q$ is clearly well defined and linear. 
The property $\chi_q(V_1\otimes V_2)=\chi_q(V_1)\chi_q(V_2)$ follows from 
\eqref{h}. 

Suppose $\chi_q(V_1)=\chi_q(V_2)$. Let $\bm$ be a monomial in $\chi_q(V_1)$ 
such that for any other monomial $\bm'$ in $\chi_q(V_1)$, the monomial 
$\varpi(\bm'/\bm)$ does not belong to $\Z[\sfq^{\alpha_i}]_{i\in I}$. 
Choose vectors in $V_1$ and $V_2$ corresponding to $\bm$ which are eigenvectors of $\phi_i^+(z)$. 
Then these vectors are clearly singular. Therefore both $V_1$ and $V_2$ contain a subquotient 
module 
which is isomorphic to $L(\bm)$. We quotient $V_1,V_2$ in Grothendieck ring by $L(\bm)$ and repeat the argument. 
It follows that $V_1$ and $V_2$ give the same class in 
$\Rep\,U_q\bo$, and thus, $\chi_q$ is injective.
\end{proof}

For a subset $J\subset I$, let $U_q\bo_J=\langle e_j,\ k^{\pm}_j\mid j\in J \rangle$ 
denote the corresponding subalgebra of $U_q\bo$.
We have the restriction functor ${res}_J:\ \mc O_\bo\to\mc O_{\bo_J}$. 
We define the corresponding ring homomorphism 
{$res_J: \mc X\to \mc X_J=\Z[[A_{j,a}^{-1}]]
[X_{j,c}^{\pm1}, y_j^b]_{j\in J, a,c\in\C^\times,b\in\C}$ 
sending 
$X_{i,a}\mapsto X_{i,a}$, $y_{i}^b\mapsto y_{i}^b$ if $i\in J$ and $X_{i,a}\mapsto 1$, 
$y_{i}^b\mapsto 1$ if $i\not\in J$. Then we clearly have
\begin{align*}
\chi_q(res_J(V))=res_J(\chi_q(V)), \qquad V\in \Ob \cO_\bo.
\end{align*}

\subsection{Examples} 
The simplest examples are the one dimensional modules. 
For $\bs b=(b_1,\dots,b_n)\in\C^I$, we set $y^{\bs b}=\prod_{i\in I}y_i^{b_i}$. Then we have
 $\chi_q(L(y^{\bs b}))=
\chi(L(y^{\bs b}))=y^{\bs b}$.

Important examples of $U_q\bo$ modules and their $q$-characters are provided by restrictions 
of $U_q\g$ modules, see Lemma \ref{lem:simple-to-simple} above. 
For $V\in\mc O_\g$, we define $\chi_q(V)$ to be the $q$-character of $V$ considered as a $U_q\bo$ module.

By \cite{FM1}, Theorem 4.1, the $q$-characters of finite-dimensional modules of 
$U_q\g$ are polynomials in $Y_{i,a}^{\pm1}$ of the form $\bm^+(1+\sum_{j}\bm_j)$ where 
$\bm^+$ is a monomial in variables $Y_{i,a}$, $i\in I, a\in \C^\times$, 
cf. Theorem \ref{Ch-Pr}, and all $\bm_j$ are monomials in $A^{-1}_{i,a}$, 
 $i\in I, a\in \C^\times$. 
The finite-dimensional modules $L(Y_{i,a})\in\mc O_\g$ are called $U_q\g$ fundamental modules.

In general, the $q$-characters of finite-dimensional modules of 
$U_q\g$ are difficult to describe, but in some cases they are known.

\medskip

{\it Example.}\ {\it The $U_q \widehat{\mathfrak {sl}}_{n+1}$ evaluation modules.} 
Let $\la=(\la_1\geq \la_2\geq\dots\geq \la_{n}\geq 0)$ be a partition with at most 
$n$ parts and let $\la'$ be the dual partition. 
A box $\Box$ is a pair of positive 
integers 
$(i(\Box), j(\Box))$. We say $\Box\in\la$ if $\la_{i(\Box)}\geq j(\Box)$. 
Define the content of a box by $c(\Box)=j(\Box)-i(\Box)$. Let $\mc T(\la)$ 
be the set of semistandard Young tableaux of shape $\la$. For $\Box\in\la$ and 
$T\in\mc T(\la)$, we have $T(\Box)\in\{1,\dots,n+1\}$.

Consider the $U_q\mathfrak {sl}_{n+1}$ 
irreducible module with highest 
$\ell$-weight corresponding to a partition $\la$. Then the corresponding 
$U_q\widehat{\mathfrak{sl}}_{n+1}$ 
evaluation module (with an appropriate choice of 
the evaluation homomorphism) has the highest monomial 
$\bm_{\la,a}^+=\prod_{j=1}^{\la_1} Y_{\la_j',aq^{2j-\la_j'-1}}$. 
Then, the $q$-character is given by (cf. \cite{FM2}, Lemma 4.7)
\begin{align}\label{la char}
\chi_q\bigl(L(\bm_{\la,a}^+)\bigr)=\bm^+_{\lambda,a}(\sum_{T\in\mc T(\la)}
\prod_{\Box\in\la}\prod_{s=i(\Box)}^{T(\Box)-1}A_{s,aq^{2c(\Box)+s}}^{-1}).
\end{align}

\medskip

{\it Example.} The $q$-characters of fundamental modules 
$\chi_q\bigl(L(Y_{i,a})\bigr)$ 
are known (the answer is very large for, say, $E_8$ type). 
See, for example, \cite{FR} for the classical series. 
Here we write a few top terms which we will need later.
\begin{align}\label{top terms}
\chi_q(L(Y_{i,a}))=Y_{i,a}(1+A_{i,aq_i}^{-1}+A_{i,aq_i}^{-1}\sum_{j\in I,\ C_{j,i}<0} 
A_{j,aq_iq_{j,i}^{-1}}^{-1}
+\dots),
\end{align}
where the dots denote terms which contain products of at least 
three $A^{-1}_{j,b}$'s.

\medskip

More generally, by \cite{MY}, Corollary 3.10,  
the $q$-character of an irreducible module $V\in \Ob \cO_\g$ 
has the form
\begin{align}\label{1+A}
\chi_q(V)=\bm^+(1+\sum_{j}\bm_j),\qquad \bm^+=\prod_{i\in I}\prod_{j=1}^{l_i}\frac{X_{i,b_{ij}}}{X_{i,a_{ij}}},
\end{align}
where all $\bm_j$ are monomials in the $A^{-1}_{i,a}$ with 
$i\in I, a\in \C^\times$, and $a_{ij}, b_{ij}\in\C^\times$. 
In particular, all generalized eigenvalues of $\phi_i^+(z)$ are rational functions in 
$\mfr_\g$.

Consider the irreducible $U_q\g$ module with highest monomial 
\begin{align*}
\bm^+_{i,a,K}=X_{i,aq_i^{K}}X^{-1}_{i,a},
\end{align*}
where $i\in I$, $a\in\C^\times$, $K\in\C$, $K\not\in2\Z_{\leq 0}$. 
\medskip

{\it Example.} {\it The $U_q \widehat{\mathfrak {sl}}_{n+1}$ parabolic Verma evaluation modules.} 
In the $U_q \widehat{\mathfrak {sl}}_{n+1}$ case
the $q$-character of $L(\bm^+_{i,a,K})$ can be computed using \eqref{la char} and Corollary 5.6 in \cite{MY}. 
We describe the result.

Consider a strip $S_i=\{(l,j)\mid l\in \{1,\dots,i\}, j\in\Z_{\geq 1}\}\subset\Z^2$. 
A plane partition of height at most $h$ 
over $S_i$ is a map $T: S_i \to \{0,1,\dots,h\}$ 
which 
is zero for all but finitely many points in $S_i$ and which has the property $T(l,j)\geq T(l+1,j)$ and 
$T(l,j)\geq T(l,j+1)$. 
Let $\mc T_{i,n+1}$ be the set of all plane 
partitions over $S_i$ of height at most $n+1-i$. 
We have
\begin{align}\label{sl M-}
\chi_q(L(\bm^+_{i,a,K}))=\bm^+_{i,a,K}
\Bigl(\sum_{T\in\mc T_{i,n+1}}\prod_{(l,j)\in S_i}
\prod_{s=0}^{T(l,j)-1}A_{i-l+1+s,aq^{-2j+l+s+1}}^{-1}\Bigr).
\end{align}
In particular, note that the dependence on $K$ is only through the monomial $\bm^+_{i,a,K}$.
\medskip

For general $\mathfrak g$ the $q$-character of $L(\bm^+_{i,a,K})$ is not known in a closed form, though one can explicitly write an arbitrary number of top terms using the algorithm of \cite{FM1}. 
We denote $\bar\chi_i$ the corresponding 
normalized character:
\begin{align}\label{chi_i}
\bar\chi_i=
y_i^{K/2}\chi(L(m^+_{i,a,K}))
=1+\sfq^{-\al_i}+ \sum_{j,C_{j,i}\neq 0} \sfq^{-\al_i-\al_j}+\dots\ ,
\end{align}
where the dots denote terms which are product of at least 
three $\sfq^{-\al_j}$'s. 
The explicit formula for  $\bar\chi_i$ was conjectured in \cite{MY}, Conjecture 6.3:
\begin{align*}
\bar\chi_i=\prod_{\al\in \Delta^+}\frac{1}{(1-\sfq^{-\al})^{\langle\omega_i^\vee,\al\rangle}}.
\end{align*}
For type $A,B,C,D$ this is a consequence of known identities 
in \cite{HKOTY}. 
It was shown in type $G_2$ in \cite{LN}.

\medskip

Finally, we discuss the $q$-characters of positive and negative fundamental 
$U_q\bo$ modules $M_{i,a}^\pm$. These modules are not restrictions of $U_q\g$ modules. 
\medskip

{\it Example.} {\it The fundamental modules $M_{i,a}^\pm$.} 
The negative fundamental module $M_{i,a}^-$ 
is obtained as a limit of appropriate $U_q\g$ modules, see \cite{HJ}. 
Its $q$-character is given by 
\begin{align}\label{M-char}
\chi_q(M_{i,a}^-)=
X_{i,aq^K}^{-1}\ \chi_q(L(\bm^+_{i,a,K}))\,.
\end{align}
Note that the right hand side is independent of $K$. 
In particular, in the case of type A, it can be written explicitly 
using formula \eqref{sl M-}.

The positive fundamental module $M_{i,a}^+$ is constructed as the dual of
the lowest weight version of  $M_{i,a}^-$ (see Lemma \ref{dual lem}). 
Its $q$-character was obtained in \cite{HJ} in a special case and 
in \cite{FH} in general:
\begin{align}\label{qch-M+}
\chi_q(M_{i,a}^+)=X_{i,a}\bar\chi_i\,.
\end{align}
It is instructive to write \eqref{M-char}, \eqref{qch-M+} as 
\begin{align*}
&\chi_q(M_{i,a}^-)=\lim_{q^{K}\to 0} X_{i,aq_i^K}^{-1}
\ \chi_q\bigl(L(X_{i,aq_i^{K}}X^{-1}_{i,a})\bigr)\,,
\\
&\chi_q(M_{i,a}^+)=\lim_{q^{-K}\to 0} 
X_{i,aq_i^{-K}}\ 
\ \chi_q\bigl(L(X_{i,a}X^{-1}_{i,aq_i^{-K}})\bigr)\,,
\end{align*}
where we impose formally the rule 
$\ds{\lim_{a\to 0} A_{j,a}=\sfq^{\alpha_j}}$. 

Quite generally, we call a module $V\in\Ob\cO_{\bo}$ {\it $s$-finite} if the 
set of currents $\{k_i^{-1}\phi_i^+(z)\}_{i\in I}$ 
has $s$ different joint eigenvalues. 
We call $V$ {\it finite type} module if it is $s$-finite for some $s\in\Z_{>0}$.

All finite-dimensional modules $V$ are at most $d$-finite, where $\dim V=d$.  
A restriction of a $U_q\g$ module $V\in\Ob\cO_\g$ 
is finite-type if and only if $\dim V<\infty$. 
Formula \eqref{qch-M+} says that the positive fundamental module 
$M^+_{i,a}$ is infinite-dimensional and $1$-finite. 
In contrast, the negative fundamental module $M^-_{i,a}$ is not finite type.

A subquotient of an $s$-finite module 
is at most $s$ finite. 
If $V_i$ is $s_i$-finite ($i=1,2$), then 
$V_1\otimes V_2$ is at most $s_1s_2$-finite, 
and $V_1\oplus V_2$  is at most $(s_1+s_2)$-finite. 
A tensor product (resp. direct sum) is finite type 
if and only if all factors (resp. summands) are finite type. 
In particular, an arbitrary
tensor product of $M^+_{i,a}$ is $1$-finite. 

\section{Modules of the form $V\otimes M$}\label{MV sub}
In this section we study $U_q\bo$ modules which are tensor products of a restriction module $V$ and a module with polynomial highest 
$\ell$-weight $M$.

\subsection{$U_q\bo$ modules with polynomial highest $\ell$-weight}

Let $M$ be an irreducible $U_q\bo$ module.  
We say that $M$ has polynomial highest $\ell$-weight if  
\begin{align}
&M=L(\bs{\Psi}^M)\,,\quad  
\bs{\Psi}^M\in \mfr_\bo\cap \C[z]^I.
\label{poly-M}
\end{align}
In this subsection we investigate special properties of such modules.
We write \eqref{poly-M} also in the monomial notation as 
\begin{align}
&M=L(\bm_p)\,,\quad  
\bm_p=m(\bs{\Psi}^M)\in\Z[X_{i,a},y_i^b]_{i\in I, a\in\C^\times,b\in\C}\,.
\label{poly-M2}
\end{align}
We denote by $\ket{\emptyset}_M$ the highest $\ell$-weight vector of $M$. 

First we show that $M$ is $1$-finite. For that purpose we need
\begin{lem}\cite{FH}\label{product of M} 
Any tensor product of positive fundamental $U_q\bo$ modules 
is irreducible. Similarly, any tensor
product of negative fundamental $U_q\bo$ modules is irreducible.
\end{lem}
\begin{proof}
We give a proof different from the one given in \cite{FH}. 

Set $M_0=M_{i_1,a_1}^+\otimes\dots\otimes M_{i_r,a_r}^+$,
$\bm=\prod_{j=1}^r X_{i_j,a_j}$. 
This module is $1$-finite.
The dual module $M_0^*$ is isomorphic to 
$M_{i_r,a_r}^{-,\vee}\otimes\dots\otimes M_{i_1,a_1}^{-,\vee}$ 
where 
$M_{i,a}^{-,\vee}=(M_{i,a}^+)^*=L^\vee(X_{i,a}^{-1})\in\mc O_\bo^\vee$. 
By the dual versions of \eqref{M-char} and \eqref{1+A}, 
the multiplicity of the monomial $\bm^{-1}$ in $\chi_q(M_0^*)$ is one. 

Suppose that $M_0^*$ contains a non-zero proper submodule $N^*$. 
Then either $N^*$ or $M_0^*/N^*$ has a singular vector 
whose $\ell$-weight $\bn^{-1}$
differs from $\bm^{-1}$ by a non-trivial monomial of the $A^{-1}_{j,a}$'s.
Let $N\subset M_0$ be the orthogonal complement of $N^*$.
From Lemma \ref{dual lem}, we conclude that either $M_0/N$ or $N$ has the $\ell$-weight $\bn$. This contradicts to the fact that $M_0$ is $1$-finite. 
Therefore $M^*_0$ is irreducible. Hence $M_0$ is also irreducible. 
\end{proof}

\begin{cor}\label{cor:1-finite}
Any module with polynomial highest $\ell$-weight 
is isomorphic to a tensor product of several positive fundamental modules and a one-dimensional module.
In particular it is $1$-finite.
\end{cor}
\begin{proof}
This follows from Lemma \ref{product of M} and \eqref{qch-M+}. 
\end{proof}

One of the basic properties of modules with polynomial highest $\ell$-weight is
the following polynomiality of currents. Introduce the notation for half currents
\begin{align*}
x^+_{i,\geqslant}(z)=\sum_{m\ge0}x^+_{i,m}z^{m}\,,
\quad 
x^-_{i,>}(z)=\sum_{r>0}x^-_{i,r}z^{r}\,.
\end{align*}

\begin{lem}\label{polyn act}
Let $M$ be as in \eqref{poly-M}.
Then for all $v\in M$ and $i\in I$ we have
\begin{align}
x^+_{i,\geqslant}(z)v\,,\ x^-_{i,>}(z)v\in M\otimes \C[z],
\quad
\phi^+_i(z)v\in \Psi_i(z)\cdot M\otimes \C[z]. 
\label{pol-cur}
\end{align}
\end{lem}
\begin{proof}
We prove equivalent statements for the (restricted) right dual module $M^*$. 
Technically it is simpler, 
because $M^*=v_0^*\cdot U^+_q\bo$, where $v^*_0$ is  the lowest $\ell$-weight vector,  
and $U^+_q\bo$ is generated by $\{x^+_{i,m}\}_{i\in I,m\ge0}$. 

First consider the vector $v_0^*$. 
We have $v^*_0 x^-_{i,m}=0$ for $i\in I, m>0$ and $v^*_0\phi^+_i(z)=v^*_0\Psi_i(z)$. 
If $m>\deg \Psi_i(z)$, then we have also 
$v^*_0 x^+_{i,m}x^-_{j,r}\in \delta_{i,j}\C v^*_0\phi^+_{i,m+r}=0$ for $r>0$. 
The other generators of $U^-_q\bo$ also kill $v^*_0 x^+_{i,m}$ for 
the weight reason.
Since $M^*$ is simple, we must have $v^*_0 x^+_{i,m}=0$ ($m>\deg \Psi_i(z)$).  

By induction on the weight, suppose that 
$w^{*}\in M^*$ satisfies $w^{*} x^\pm_{i,m}=0$ for $m$ large enough and $w^{*}\phi^+_i(z)\in M^*\otimes\C[z]\Psi_i(z)$
for all $i\in I$.  
We show that $v^*=w^{*} x^+_{j,n}$ ($j\in I, n\ge0$) has the same property. 

For $m$ large, we have $v^*x^-_{i,m}\in\delta_{i,j}\C w^*\phi^+_{i,m+n}=0$. 
As for $\phi^+_i(z)$, we can use the following relation which follows from
the quadratic relation \eqref{phi-x}
\begin{align*}
x^+_{j,n}\phi^+_{i}(z)\in \sum_{r=0}^p\C z^r\phi_i^+(z)x^+_{j,n+r}+\C x^+_{j,n+p}z^p\phi^+_i(z)\,,
\end{align*}
where $p\ge1$ is arbitrary. 
Applying this to $w^*$ and choosing $p$ large enough, we find $v^{*}\phi^+_i(z)\in M^*\otimes\C[z]\Psi_i(z)$.  
Similarly we can show $v^* x^+_{i,m}=0$ for $m$ large using the quadratic relation \eqref{x-x}. 
\end{proof}
\medskip

\begin{cor}\label{e-delta-zero}
For any $v\in M$ and $\alpha\in\Delta^+$ 
we have $e_{p\delta\pm\alpha}v=0$ for sufficiently large $p$.
\end{cor}
\begin{proof}
 This follows from Lemma \ref{polyn act} and Lemma \ref{lem:root-vec}. 
\end{proof}
\medskip

\subsection{Submodules of modules of type $V\otimes M$}\label{MV subsec}
Let now $V$ be an irreducible $U_q\g$ module with highest $\ell$-weight $\bs \Psi^V\in\mfr_{\g}$, 
\begin{align*}
V=L(\bm_0)\,,\quad  \bm_0=m(\Psib^V)\,.
\end{align*}
The following lemma describes the structure of the action of $x^\pm_i(z)$ in $V$.
\begin{lem}(\cite{Y},Proposition 3.1)\label{delta lem} 
Let $V\in\mc \Ob \cO_\g$, $\bm$ a 
monomial in $\chi_q(V)$ and $v\in V_\bm$. 
Then the formal series $x^{\pm}_i(z)v$ has the form 
\begin{align*}
x^{\pm}_i(z)v=\sum_{a}\sum_{k=0}^{r_a}v_{k,a} (\partial_a)^k\delta(za), 
\quad v_{k,a}\in V_{\bm A_{i,a}^{\pm1}}\,, 
\end{align*}
where $a$ runs over a finite subset of $\C^\times$ 
and $r_a=\dim V_\bm+\dim V_{\bm A_{i,a}^{\pm1}}-2$.
\qed
\end{lem}

We study the properties of the module $V\otimes M$.
To this aim, it is convenient to use the Drinfeld coproduct \eqref{Drin coprod}.
\begin{lem}\label{lem:Drinfeld-copro}
On $V\otimes M$, the Drinfeld coproduct $\Delta_D$ gives a well defined structure of 
a $U_q\bo$ module which we denote  $V\otimes_D M$. 
As $U_q\bo$ modules, $V\otimes_D M$ and $V\otimes M$ are isomorphic. 
\end{lem}
\begin{proof} 
We show that $\sigma(\cR_+)$ is a well-defined linear operator on $V\otimes M$. 
Consider its action on a weight vector $v\otimes w\in V\otimes M$. 
Expanding \eqref{R+}, we obtain a linear combination of terms 
of the form 
\begin{align*}
f_{-k_1\delta-\beta_1}^{n_1}\cdots f_{-k_N\delta-\beta_N}^{n_N}v
\otimes 
e_{k_1\delta+\beta_1}^{n_1}\cdots e_{k_N\delta+\beta_N}^{n_N}w\,,
\end{align*}
where $k_i,n_i\ge0$ and $\beta_i\in\Delta^+$, $1\le i\le N$. 
The second component stays in a finite dimensional subspace 
of $M$ of weight $\ge \wt w$. 
For this term to be non-zero, there are only finitely many choices of $(n_i,\beta_i)$. 
From Corollary \ref{e-delta-zero}, 
it is non-zero only for finitely many $k_i$'s. 
Therefore $\sigma(\cR_+)v\otimes w$ comprises only finitely many non-zero terms.
In view of Proposition \ref{prop:twist}, 
$\Delta_D(x)$ has a well-defined action on $V\otimes M$
for any $x\in U_q\bo$, and 
$\sigma(\cR_+)$ gives an intertwiner between 
$V\otimes_D M$ and $V\otimes M$. 
\end{proof}

The next lemma shows that submodules of $V\otimes_D M$ have a very special form.
\begin{lem}\label{submodules}
Let $W\subset V\otimes_D M$ be a non-zero submodule. 
Then there exists a linear subspace $V^{(0)}\subset V$ 
such that $W=V^{(0)}\otimes M$. 
The highest $\ell$-weight vector $v_0$ of $V$ belongs to $V^{(0)}$. 
\end{lem}
\begin{proof}
Let $w\in W$ be a vector in an $\ell$-weight subspace. 
Since $W$ is a submodule, 
\begin{align*}
\Delta_D(x^-_{i,k})w=
\bigl(\sum_{j\ge0}x^-_{i,k-j}\otimes \phi^+_{i,j}\bigr)w
+(1\otimes x^-_{i,k})w 
\end{align*}
belongs to $W$ for any $i\in I$, $k>0$. 
By Lemma \ref{delta lem}, the first term is a sum of terms 
which belong to $\ell$-weight subspaces different from that of $w$. 
On the other hand, by Corollary \ref{cor:1-finite}, 
the second term belongs to the same $\ell$-weight subspace as 
that of $w$. 
Hence both terms separately belong to $W$. Since $w$ is arbitrary,
we conclude that
\begin{align}
\bigl(\sum_{j\ge0}x^-_{i,k-j}\otimes \phi^+_{i,j}\bigr)W\subset W\,, 
\quad
(1\otimes x^-_{i,k}) W\subset W\,.
\label{xW-W1}
\end{align}

Let now $x\in U_q^-\bo$. Then we claim that $\Delta_D(x)=1\otimes x+\cdots$
where $\cdots$ is a sum of terms whose first component contains
at least one $x^-_{i,j}$. Indeed, the element 
$x\in U_q^-\bo\subset\U^-_q\g$
can be written in terms of $x^-_{i,m}$, $i\in I$, $m\in\Z$. We have
$\Delta_D(x^-_{i,m})=1\otimes x^-_{i,m}+\cdots$ and 
$U_q\bo$ is a coideal, see \eqref{coideal-b}.
Therefore the claim follows.
 
Hence, we can generalize \eqref{xW-W1} to
\begin{align}
(1\otimes x)W\subset W\quad (x\in U^-_q\bo).  
\label{xW-W2}
\end{align}
By the same argument leading to \eqref{xW-W1}, we have also
\begin{align}
(x^+_{i,k}\otimes 1) W\subset W\,,
\quad
\bigl(\sum_{j\ge0}\phi^-_{i,-j}\otimes x^+_{i,k+j}\bigr)W\subset W\,.
\label{xW-W3}
\end{align}

Consider the linear subspace
$V^{(0)}=\{v\in V\mid v\otimes\ket{\emptyset}_M\in W\}$. 
Using \eqref{xW-W2}, we obtain that $V^{(0)}\otimes M\subset W$. 
We prove the equality by showing the following statement:
If $w=\sum_{r=1}^Nv_r\otimes m_r\in W$ 
and $\{v_r\}_{r=1}^N\subset V$, 
$\{m_r\}_{r=1}^N\subset M$ are linearly independent
weight vectors,
then $v_r\in V^{(0)}$ for all $r$. 

Suppose $N=1$, so that $w=v_1\otimes m_1$, $v_1\neq0$, $m_1\neq0$. 
If $m_1\in\C\ket{\emptyset}_M$, there is nothing to show. 
Otherwise there exists an $i\in I$ and $k\ge0$ such that $x^+_{i,k}m_1\neq0$. 
By Lemma \ref{polyn act}, there is the largest $k$ with this property. 
For this $k$ we obtain $\phi^-_{i,0}v_1\otimes x^+_{i,k}m_1\in W$
by using \eqref{xW-W3}.  
Repeating this process, we arrive at $v_1\otimes\ket{\emptyset}_M\in W$. 

Next let $N>1$, and assume that the statement is proved for $N'<N$. 
Arguing similarly as above, we can find an $i\in I$ and $k\ge0$ such that
$x^+_{i,k}m_r\neq 0$ for some $r$ and that $x^+_{i,l}m_s=0$ for all
$l>k$, $1\le s\le N$. Applying \eqref{xW-W3}, we obtain that 
$\sum_{r=1}^N \phi^-_{i,0}v_r\otimes x^+_{i,k}m_r\in W$.  
If $\{ x^+_{i,k}m_r\}_{r=1}^N$ is linearly independent, 
then we repeat this procedure. 
After a finite number of steps, we obtain vectors $m'_r\in M$ 
where $\sum_{r=1}^N v_r\otimes m'_r\in W$,   
$m'_r\neq 0$ for some $r$, 
and $\{m'_r\}_{r=1}^N$ is not linearly independent. 
Renumbering indices, we may assume that 
$\{m'_s\}_{s=1}^{N'}$  ($0<N'<N$) 
is linearly independent and that 
$m'_r=\sum_{s=1}^{N'}m'_s a_{s,r}$ ($N'<r\le N$) with some
$a_{s,r}\in \C$. Then
\begin{align*}
\sum_{s=1}^{N'}v'_s\otimes m'_s\, \in W\,,
\quad v'_s=v_s+\sum_{r=N'+1}^Nv_ra_{r,s}\,. 
\end{align*}
It follows that $v'_s\in V^{(0)}$ ($1\le s\le N'$) by the induction hypothesis. 
Since $V^{(0)}\otimes M\subset W$, we see that 
$\sum_{s=1}^{N'}v'_s\otimes m_s$ belongs to $W$. This in turn implies
that
\begin{align*}
\sum_{r=N'+1}^Nv_r\otimes\bigl(m_r-\sum_{s=1}^{N'}m_s a_{s,r}\bigr) 
\ \in W\,.
\end{align*} 
Using again the induction hypothesis, we conclude that
$v_r\in V^{(0)}$ for all $r$. 

Finally, since $V^{(0)}\neq0$, we can use \eqref{xW-W3} 
to show that $v_0\in V^{(0)}$.
\end{proof}

\begin{cor}\label{inf dim cor} 
Let $\bm_0=m(\bs{\Psi}^V)$, $\bs{\Psi}^V\in\mfr_\g$, and 
let $\bm_p$ be a monomial in $\Z[X_{i,a},y_i^b]_{i\in I, a\in\C^\times,b\in\C}$. 
If $\bm_p\not\in\Z[y_i^b]_{i\in I, b\in\C}$, then the module 
$L(\bm_0\bm_p)$ is infinite-dimensional. \qed
\end{cor}
\begin{proof}
Module $L(\bm_0\bm_p)$ is a subquotient of $L(\bm_0)\otimes_D L(\bm_p)$.
Lemma \ref{submodules} says that it is actually a submodule of the form $V^{(0)}\otimes L(\bm_p)$. 
Hence  $L(\bm_0\bm_p)$ is finite dimensional only if $L(\bm_p)$ is one-dimensional.
\end{proof}
 
Next we give a sufficient condition for $V\otimes M$ to be irreducible. 
The following lemma shows that 
a cancellation in $q$-characters {\it must} happen 
in order for $V\otimes M$ to be reducible. 
\begin{lem}\label{irred lemma}
Let 
\begin{align*}
&\chi_q(V)=\bm_0(1+\sum_s\prod_{t}A_{i'_{t,s},a'_{t,s}}^{-1})\,,
\quad
\chi_q(M)=y^{\bs b}\prod_{j=1}^r X_{i_j,a_j}\cdot \bar\chi_{i_j}\,.
\end{align*}
If $(i_j,a_j)\neq (i'_{t,s},a'_{t,s})$ for all $j$ and all $t,s$,  
then $V\otimes M$ is irreducible.
\end{lem}
\begin{proof} 
We again work with $V\otimes_D M$.
Let $W$ be a non-zero submodule of  $V\otimes_D M$. 
By Lemma \ref{submodules}, it has the form $V^{(0)}\otimes M$ 
where $V^{(0)}=\{v\in V\mid v\otimes\ket{\emptyset}_M\in W\}$. 
By \eqref{xW-W3}, we have $x^+_{i,p}V^{(0)}\subset V^{(0)}$ ($i\in I$, $p\ge0$).
We show that $x^-_{i,p}V^{(0)}\subset V^{(0)}$ ($i\in I$, $p>0$). 

To this aim, let $v\in V^{(0)}\cap V_\bm$. By Lemma \ref{delta lem}, we can write
$x^{-}_i(z)v=\sum_{k,a}v_{k,a} \partial_a^k\bigl(\delta(za)\bigr)$, 
where $v_{k,a}\in V_{\bm A_{i,a}^{-1}}$. 
Then 
\begin{align*}
\Delta_D(x_{i,>}^-(z))(v\otimes \ket{\emptyset}_M)
\equiv
\sum_{k,a}{\Big(}\partial_a^k\bigl(\delta(za)\bigr)
\prod_{j,\ i_j=i}(1-za_j){\Big)}_{>}
(v_{k,a}\otimes \ket{\emptyset}_M)\quad
\bmod V^{(0)}\otimes M\,.
\end{align*}
Here, for a formal series $r(z)=\sum_{n\in\Z} r_nz^n$,  
we set $r(z)_{>}=\sum_{n> 0}r_nz^n$. 
If $v_{k,a}\neq 0$, then by the assumption
 $a_j\neq a$ for all $j$ in the product. 
Hence $v_{k,a}\in V^{(0)}$. 

From the proof of Lemma \ref{lem:simple-to-simple}, 
it follows that  $x^\pm_{i,p}V^{(0)}\subset V^{(0)}$ for all $i\in I$ and 
$p\in \Z$. 
Since $V$ is irreducible, we have $V^{(0)}=V$ and hence 
$W=V\otimes M$. 
This shows that $V\otimes_D M$ is irreducible. 
\end{proof}

We finish the section with a technical lemma which is useful in the study of submodules of $V\otimes_D M$. The algebra $U_q^-\bo$ is not generated by 
$x_{i,k}^-$ with $k>0$, however, in many cases, 
it is possible to avoid checking the invariance 
of the submodule with respect to the other generators.

We denote the highest $\ell$-weight vector of $V$ by $v_0$ 
and that of $M$ by $\ket{\emptyset}_M$.

\begin{lem}\label{other generators lem} Assume that $V_0\subset V$ is 
an $\ell$-weighted subspace such that the space $V_0\otimes_D M$ is 
invariant under the action of 
$\{x^+_{i,k}\}_{i\in I, k\ge0}$ and $\{x^-_{i,k}\}_{i\in I, k>0}$.
Assume that $v_0\in V_0$ 
and 
$V_0\otimes_D \ket{\emptyset}_M
\subset U_q\bo(v_0\otimes_D \ket{\emptyset}_M)$. 
Finally, assume that $V_0$ and the $\ell$-weighted 
complement of $V_0$ in $V$ have no common $\ell$-weights. 
Then $V_0\otimes_D M$ is a submodule in $V\otimes_D M$.
\end{lem}
\begin{proof}
Let $W$ be the submodule generated by $V_0\otimes_D M$. 
The module $W$ is the cyclic submodule of $V\otimes_D M$ generated by 
$v_0\otimes \ket{\emptyset}_M$.
Indeed, by assumption, 
$V_0\otimes_D \ket{\emptyset}_M\subset U_q\bo(v_0\otimes_D 
\ket{\emptyset}_M)$ 
and we have \eqref{xW-W2}. 
We wish to show $W=V_0\otimes_D M$.

By Lemma \ref{polyn act}, and since $V_0\otimes_D M$ is invariant
by $\{x^-_{i,k}\}_{i\in I, k>0}$, there exists a $K>0$ such that 
$x^-_{i,k} (V_0\otimes \ket{\emptyset}_M)\subset V_0\otimes \ket{\emptyset}_M$ 
for all $i\in I$ and $k>K$. Consider the 
following subalgebras of $U_q\g$,
\begin{align*}
E_K=\langle x^-_{i,k}, i\in I,k>K\rangle\cdot U^0_q\bo\cdot U^+_q\bo\,,
\qquad 
E=\langle x^-_{i,k},  i\in I, k\in \Z\rangle\cdot U^0_q\bo\cdot U^+_q\bo\,.
\end{align*}
Note that we have the inclusions of subalgebras 
$E_K\subset U_q\bo \subset E$.

Then $V_0\otimes \ket{\emptyset}_M$ is an $E_K$-module. 
Consider the $E$ module 
$\bar W=E\otimes_{E_K}(V_0\otimes \ket{\emptyset}_M)$ 
induced from the $E_K$-module $V_0\otimes \ket{\emptyset}_M$. 
Clearly $W$ is a 
subquotient of $\bar W$ considered as $U_q\bo$ module. 
Therefore, to prove the lemma it is sufficient to show 
that all $\ell$-weights of $\bar W$ effectively coincide 
with the $\ell$-weights of $V_0\otimes \ket{\emptyset}_M$. 
More precisely, we show that for each $\ell$-weight 
vector $\bar w\in\bar W$ there exists an $S_{\bar w}>0$ and 
$v_{\bar w}\in V_0$ such that the eigenvalues of $h_{i,s}$ 
with $i\in I$ and $s>S_{\bar w}$ on $\bar w$ equal to the corresponding 
eigenvalues of $h_{i,s}$ on 
$v_{\bar w}\otimes \ket{\emptyset}_M$.

Define a filtration 
$\bar W_0\subset\bar W_1\subset \bar W_2\subset\cdots$ of $\bar W$, 
where $\bar W_0=V_0\otimes \ket{\emptyset}_M$, and  
$\bar W_r$ ($r\ge1$) is spanned by $\bar{W}_{r-1}$ and elements
\begin{align}
\bar w=x(v\otimes  \ket{\emptyset}_M)\,,  
\quad x=x^-_{j_1,m_1}\cdots x^-_{j_r,m_r}\,, 
\label{bar w}
\end{align}
with $j_t\in I$, $m_t\in\Z$, and $v\in V_0$. 
Then for any $\bar w$  of the form \eqref{bar w}, 
there exists an $S_{\bar w}>0$ such that 
$[h_{i,s},x](v\otimes \ket{\emptyset}_M)$ 
belongs to $\bar W_{r-1}$ if $s>S_{\bar w}$.  
This follows from \eqref{hx} with the help of the relation \eqref{x-x} 
and the fact that $x^-_{i,k}$ with large $k$ preserves 
$V_0\otimes \ket{\emptyset}_M$. 
Thus the action of $h_{i,s}$ on $\bar W$ is represented by a 
triangular matrix whose diagonal entries are the
eigenvalues of $h_{i,s}$ on $V_0\otimes\ket{\emptyset}_M$. 
Hence we obtain the desired result.

\end{proof}

\subsection{Grading} 
As an application of Lemma \ref{lem:Drinfeld-copro}, 
we show that the module $M$ has 
a sort of homogeneous grading.  
A similar result was proved in \cite{FH} 
(for the dual module and with $N=1$) by a different method.
\begin{prop}\label{prop:gradingM}
Let $M=M^+_{i_1,a_1}\otimes\cdots\otimes M^+_{i_N,a_N}$, and let $\Psib^M=(\Psi^M_i(z))_{i\in I}$ be its highest
$\ell$-weight.
Set $\omega^\vee=\omega^\vee_{i_1}+\cdots+\omega^\vee_{i_N}$,
$\mu=(\mu_i)_{i\in I}\in\tb^*$, $\mu_i=\prod_{j:i_j=i}a_{j}^{-1/2}$.
Then there exists a grading $M=\oplus_{m=0}^\infty M[m]$ as vector space
with the following properties.
\begin{align}
&
\C \ket{\emptyset}_M=M[0]\,,\quad
M_{\mu\sfq^{-\beta}}
=\bigoplus_{m=0}^\infty M_{\mu\sfq^{-\beta}}\cap M[m]\,
\quad (\beta\in Q^+)\,,
\label{gradingM1}\\
&e_{k\delta-\alpha}M[m]\subset M[m+k]\quad (k>0,\alpha\in\Delta^+)\,,
\label{gradingM2}\\
&e_{k\delta+\alpha}M[m]\subset \sum_{j=0}^{(\omega^\vee,\alpha)}M[m+k-j]
\quad (k\ge0,\alpha\in\Delta^+)\,,
\label{gradingM3}\\
&\overline{\phi}^+_{i,k}M[m]\subset  M[m+k]\quad 
(k\ge0,i\in I).
\label{gradingM4}
\end{align}
In the last line we set $\Psi^M_i(z)^{-1}\phi^+_i(z)=\sum_{k\ge0}\overline{\phi}^+_{i,k}z^k$.
\end{prop}
\begin{proof}
Let $r=(\theta,\theta)/2$ where $\theta$ 
is the maximal root of $\gb$, and let $r_i=r/d_i\in\{1,2,3\}$ for $i\in I$ so that $q_i^{r_i}=q^r$. 
Let 
$V_{i,a}=L(X_{i,a}^{-1}X_{i,a q_i^{-2r_i}})$. Then 
$V_{i,a}$ is the restriction of a finite-dimensional $U_q\g$ module. 
We set $V=V_{i_1,a_1}\otimes\cdots\otimes V_{i_N,a_N}$,
and denote by $v_0\in V$ the tensor product of 
highest $\ell$-weight vectors of $V_{i_j,a_j}$.  

By Lemma \ref{submodules}, a non-zero
submodule of $V\otimes_D M$ has the form $V_0\otimes M$ with 
$v_0\in V_0$. Let $S$ be the irreducible submodule containing 
$v_0\otimes\ket{\emptyset}_M$. 
Comparing highest $\ell$-weights, we must have
$S=\C v_0\otimes M$; furthermore
\begin{align*}
S\simeq M^+_{i_1,a_1 q^{-2r}}\otimes\cdots\otimes M^+_{i_N,a_N q^{-2r}} 
\simeq L(K)\otimes \tau^*M\,,
\end{align*}
where $L(K)$ is a one-dimensional module corresponding to the weight of 
$v_0$, and $\tau^*M$ denotes the twist of $M$ by the automorphism
$\tau(x)=q^{-2r\hdeg x}x$ ($x\in U_q\bo$). 
Let $\tau:\tau^*M \to M$ 
be the unique linear isomorphism such that 
$\tau\circ x=\tau(x) \circ \tau$ and 
$\tau\ket{\emptyset}_{\tau^*M}=\ket{\emptyset}_M$.  
Let $\Phi$ be the composition of the natural maps
\begin{align*}
M\longrightarrow \C v_0\otimes M \overset{\sim}{\longrightarrow}
L(K)\otimes \tau^*M
\overset{\id\otimes\tau}{\longrightarrow} L(K)\otimes M
\longrightarrow M\,
\end{align*}
sending $\ket{\emptyset}_{M}$ to $\ket{\emptyset}_M$.   
By construction $\Phi$ commutes with $k_i$, 
and we have
\begin{align}
&\Phi\circ x=q^{-2r\hdeg x}x\circ\Phi\quad (x\in U^-_q\bo),
\label{Phi x-}\\
&\Phi\circ \frac{\phi^+_i(z)}{\Psi^M_i(z)}
=\frac{\phi^+_i(q^{-2r}z)}{\Psi^M_i(q^{-2r}z)}\circ\Phi\,.
\label{Phi phi}
\end{align}
Eq. \eqref{Phi x-} implies that $\Phi$ is diagonalizable 
with eigenvalues of the form $q^{-2rm}$, $m\in\Z_{\ge0}$.
Let $M[m]$ denote the eigenspace corresponding to the eigenvalue $q^{-2rm}$. 
We have $M=\oplus_{m=0}^\infty M[m]$ and
$M[0]=\C \ket{\emptyset}_M$.
Since $\Phi$ commutes with $k_i$, \eqref{gradingM1} is clear. Also 
\eqref{gradingM2}, \eqref{gradingM4} are obvious from the properties 
\eqref{Phi x-} and \eqref{Phi phi}, respectively.
Note that \eqref{gradingM4} implies 
\begin{align}
&\phi^+_{i,k}M[m]\subset \sum_{j=0}^{(\omega^\vee,\alpha_i)}M[m+k-j]\,\quad
(k\ge0,i\in I).
\label{gradingM5}
\end{align}
It remains to prove \eqref{gradingM3}. 
Let $\Psi^{M,-}_i(z)$ be the expansion of $\Psi^M_i(z)$ at $z=0$.
Consider the series 
\begin{align*}
\overline{x}^+_{i,\gge}(z)=\sum_{k\ge0}\overline{x}^+_{i,k}z^{-k}
=\left(\frac{x^+_i(z)}{\Psi^{M,-}_i(z)}\right)_{\gge}\,.
\end{align*}
The coefficients have the form 
$\overline{x}^+_{i,k}=\sum_{l\ge0}c_l x^+_{i,k+l}$, 
and thanks to Lemma \ref{polyn act}
they are well defined operators on $M$. 
Similarly to \eqref{Phi phi}, we have the relation
\begin{align*}
&\Phi\circ \overline{x}^+_{i,\gge}(z)
=\overline{x}^+_{i,\gge}(q^{-2r}z)\circ\Phi\,.
\end{align*}
Arguing as in \eqref{gradingM5}, we deduce
\eqref{gradingM3} for $e_{k\delta+\alpha_i}=x^+_{i,k}$. 
Since these elements generate $U_q^+\bo$, the proof is now complete. 

\end{proof}

\section{Finite type modules}\label{finite sec}
\subsection{Definition of finite type modules}

We denote $\Fin\subset \mc O_\bo$ the category of 
all highest $\ell$-weight modules of finite type. 
We denote 
$\Rep_F\, U_q\bo\subset \Rep\,U_q\bo$ 
the corresponding Grothendieck ring.

If $J\subset I$, then the restriction map preserves the 
property of being finite type, and we have $res_J: \Fin\to \cO^{fin}_{\bo_J}$ 
and 
$res_J: \Rep_F\ U_q\bo \to \Rep_F\ U_q\bo_J$.

\subsection{Classification of finite type modules}
For $i\in I$, we call a monomial $\bm\in \mc X$ {\it $i$-dominant} 
if $\bm\in \Z[Y_{i,a},X_{i,a}, y_i^{b}, 
X_{j,a}^{\pm1}, y_j^{b}]_{j\in I\backslash\{i\},a\in \C^\times,b\in\C}$. 
We call a monomial $\bm\in \mc X$ {\it dominant} if it is 
$i$-dominant for all $i\in I$.

\begin{thm}\label{class thm} 
The irreducible module $L(\bm)$ 
is finite type if and only if $\bm$ is dominant.
\end{thm}
\begin{proof}
If $\bm$ is dominant, 
then write $\bm=\bm_1\bm_0\bm_p$, 
where 
$\bm_1$ is a monomial in $y_{i}^{b}$, 
$\bm_0$ is a monomial in $Y_{i,a}$, and  
$\bm_p$ is a monomial in $X_{i,a}$. 
Thus,
$L(\bm)$ is a subquotient of the tensor product 
of three finite type modules: one-dimensional $L(\bm_1)$, 
finite-dimensional $L(\bm_0)$, and $1$-finite $L(\bm_p)$. 
Therefore it is finite.

To prove the only if statement, it is sufficient to prove it 
in the case $\mathfrak{g}=\widehat{\mathfrak{sl}}_2$. 
We write
$y^{b}, X_a, Y_a$ for $y_1^{b}, X_{1,a}, Y_{1,a}$, respectively.

Let 
\begin{align}
 \bm=y^{b}\prod_{i=1}^{i_0}X_{a_i}\prod_{j=1}^{j_0}X_{b_j}^{-1}\,.
\label{yb}
\end{align}

Suppose $i_0<j_0$. Then we claim that $L(\bm)$ is not finite type. 
Indeed, considering the dual module $L(\bm)^*=L^\vee(\bm^{-1})$
 (see Lemma \ref{dual lem}) and using Corollary \ref{inf dim cor}, we 
see that $L(\bm)$ is infinite dimensional. Moreover, $L(\bm)$ is a 
subquotient of a tensor product of several negative fundamental 
modules $L(X_b^{-1})$, restriction of a $U_q\g$ module, 
and a one-dimensional module. 
All $\ell$-weight spaces in such modules are finite-dimensional, 
so $L(\bm)$ is not finite type.

Suppose $i_0=j_0$. Then $L(\bm)$ is a tensor product of 
restriction of a $U_q\g$ module 
and a one-dimensional module and theorem follows.

Finally suppose $i_0>j_0$. 
We show that there is a way to write $\bm=\bm_1\bm_0\bm_p$, 
where $L(\bm_1)$ is one-dimensional, 
$L(\bm_0)$ is the restriction of a $U_q\g$ module, 
$L(\bm_p)$ is $1$-finite, 
and $L(\bm)=L(\bm_1)\otimes L(\bm_0)\otimes L(\bm_p)$. 
The procedure is similar to writing a highest $\ell$-weight of a 
finite-dimensional $U_q \widehat{\mathfrak{sl}}_2$ module as a 
product of highest $\ell$-weights of evaluation modules. 

We call a pair $(a,b)$, $a,b\in \C^\times$, a string
 with head $a$ and tail $b$. 
The string is of finite size 
$s\in\Z_{\geq 1}$
if and only if 
$a=bq^{-2s+2}$. In other words, the restriction module $L(X_aX_{b}^{-1})$ is 
of finite dimension $s$ if and only if $(a,b)$ is of finite size $s$. 
Moreover, it is an evaluation module and we have
\begin{align*}
\chi_q (L(X_aX_{b}^{-1}))=X_aX_{b}^{-1}
(1+\sum_{r=1}^{s-1}\prod_{t=0}^{r-1}A^{-1}_{bq^{-2t}})\,,
\end{align*}
where $s$ is the size of the string.  
If the size is not finite, the same equation holds with 
$s=\infty$.

We say that a point $c$ is in a generic position 
relative to the string $(a,b)$ if 
$c\not\in bq^{-2\Z}$, or if $(a,b)$ is of finite size $s$ and 
$c\neq bq^{-2t}$, $0\le t\le s-2$. 
In other words, if $c$ 
is in a
generic position then the module 
$L(X_aX_{b}^{-1})\otimes L(X_c)$ 
is irreducible by Lemma \ref{irred lemma}.

Starting with the monomial $\bm$ \eqref{yb}
we form strings with tails $(a_{i_j},b_j)$ 
($j=1,\dots,j_0$) such that all $i_j$ are distinct and such that 
all $(i_0-j_0)$ remaining $a_i$ are in a generic position relative
to all 
$j_0$ strings. It is clear that this can be done (sometimes in several ways). 

Then we set 
\begin{align*}
\bm_1=y^{b}\,,\qquad  
\bm_0=\prod_{j=1}^{j_0} X_{a_{i_j}}X^{-1}_{b_j}\,,
\qquad  \bm_p=\prod_{i, \ a_i\not\in\{a_{i_1},\dots,a_{i_{j_0}}\}} X_{a_i}\,.
\end{align*}
Then by Lemma \ref{irred lemma} we have 
$L(\bm)=L(\bm_1)\otimes L(\bm_0)\otimes L(\bm_p)$. 
It follows that $L(\bm)$ is 
finite type if and only if $L(\bm_0)$ is finite type, 
and, therefore, if and only if $L(\bm_0)$ is finite dimensional. 
\end{proof}

\begin{cor}
The ring $\Rep_F\, U_q\bo$ is topologically generated 
by one-dimensional modules $L(y^{\bs b})$, 
restrictions of $U_q\g$ fundamental modules $L(Y_{i,a})$, and 
positive fundamental modules $M^+_{i,a}$. 
\end{cor}

A $U_q\bo$ module $V$ is called {\it prime} if $V=V_1\otimes V_2$ implies $\dim V_1=1$ or $\dim V_2=1$. 

The following follows from the proof of Theorem \ref{class thm}.

\begin{cor}\label{sl2 cor} Let ${\mathfrak{g}}=\widehat{\mathfrak{sl}}_2$.
If $L(\bm)\in \Fin$ is an irreducible module of finite type, then $L(\bm)=L(\bm_0)\otimes L(\bm_p)$ 
is a tensor product of restriction of a finite-dimensional module $L(\bm_0)$ and a 1-finite module $L(\bm_p)$,  
where $\bm_0$ and $\bm_p$ are uniquely determined by $\bm$.

In particular, $L(\bm)\in \Fin$ is prime if and only if $\bm=X_{a}$ or $\bm=\prod_{i=0}^{s-1} Y_{aq^{-2s}}$. \qed
\end{cor}

For a monomial $\bm\in \mc X$, define $d_i(\bm)$ 
by setting $d_i(y_j^{\bs b})=0$, 
$d_i(X_{j,a}^{\pm1})=\pm\delta_{ij}$
and requiring $d_i(\bm'\bm'')=d_i(\bm')+d_i(\bm'')$.
 
\begin{cor}\label{A cor} Let $V=L(\bm)\in \Fin$ 
be an irreducible highest $\ell$-weight $U_q\bo$ module of finite type. 
Then $d_i(\bm)\geq 0$ for all $i\in I$, and $\chi_q(V)$ has the form
\begin{align*}
\chi_q(V)=
\bm(1+\sum_{r=1}^N \bm_r)\prod_{i\in I} \bar\chi_i^{d_i(\bm)}, 
\end{align*}
where $\bm_r\in \Z[A^{-1}_{i,a}]_{i\in I, a\in \C^\times}$ for all $r$. 

In particular, $L(\bm)\in \Fin$ is finite-dimensional if and only if $d_i=0$ for all $i\in I$.

If $\bm$ does not depend on $y_i$, 
then $\chi_q(V)$ can be written as a Laurent polynomial
in $\chi_q(M^+_{i,a})$.
\end{cor}

As a byproduct we also recover a result of \cite{B}.

\begin{cor}
Let $W$ be a finite-dimensional $U_q\bo$ module. 
Then $W=L(y^{\bs b})\otimes V$ where $V$ is the 
restriction of a finite-dimensional $U_q\g$ module.
\end{cor}
\begin{proof}
If $W$ is finite-dimensional, then $W$ is finite type and Theorem \ref{class thm} applies. 
Moreover, by Corollary \ref{A cor}, 
we have $d_i=0$ for all $i\in I$.
\end{proof}
\subsection{$2$-finite modules}
In this section we discuss examples of non-trivial modules of finite type.

The $q$-characters of all $U_q\widehat{\mathfrak{sl}}_2$ finite type modules can be written immediately from Corollary \ref{sl2 cor}. Moreover, we have a notion of dominant monomials. Therefore the $q$-characters of many modules could be built recursively by the algorithm described in \cite{FM1}. 
However, we also have an alternative approach due to results of Section \ref{MV sub}, which we now use.

If $W$ is an irreducible module
of finite type, then by Theorem \ref{class thm} $W$ is 
(up to tensoring by a one dimensional module) 
a subquotient of  $ V\otimes M$, 
where $V$ is the restriction of a $U_q\g$ module 
and $M$ is a tensor product of positive fundamental modules. 

We start with the classification of $1$-finite modules.
\begin{prop}\label{all 1 finite}
Let $M\in\Ob\cO_{\bo}$ be an irreducible module. 
Then $M$ is  $1$-finite if and only if it has polynomial highest
$\ell$-weight.
\end{prop}
\begin{proof}
It is enough to show the `only if' part.

By Theorem \ref{class thm}, there exists a dominant monomial $\bm\in\mc X$ such that $M=L(\bm)$. 
Tensoring $M$ with an 
appropriate one dimensional module, 
we may assume that
$\bm=\bm_0\bm_p$, where $\bm_0$ is a 
monomial in the $Y_{i,a}$'s and $\bm_p$ is a monomial in the  $X_{i,a}$'s. 
Replacing $X_{i,aq_i}Y_{i,a}$ with $X_{i,aq_i^{-1}}$ as necessary,  
we may assume further that if $\bm_0$ is divisible by $Y_{i,a}$ then $\bm_p$ is not divisible by $X_{i,aq_i}$. 
Under this condition we claim that $\bm_0=1$.

Indeed, let $V=L(\bm_0)$, $M'=L(\bm_p)$. 
By Lemma \ref{submodules}, any non-zero submodule of $V\otimes_D M'$ contains
$v_0\otimes\ket{\emptyset}_{M'}$ 
where $v_0$ is the highest $\ell$-weight vector of $V$.  
Therefore $M$  
is a submodule of $V\otimes_D M'$ 
containing $v_0\otimes\ket{\emptyset}_{M'}$.  
If $\dim V>1$, then 
there exist $i\in I$ and $a\in \C^\times$ 
such that $Y_{i,a}$ divides $\bm_0$. We have
\begin{align*} 
\Delta_D\bigl(x_{i,>}^-(z)\bigr)(v_0\otimes \ket{\emptyset}_{M'})=
\Bigl(x_{i}^-(z) v_0\otimes
\phi^+_i(z)\ket{\emptyset}_{M'}
\Bigr)_{>}+v_0\otimes x_{i,>}^-(z)\ket{\emptyset}_{M'}\,.
\end{align*}
Since $X_{i,aq_i}$ does not divide $\bm_p$,  
Lemma \ref{delta lem}  implies that 
the first term in the right hand side is a non-zero vector. 
Since it is a sum of 
terms of $\ell$-weight different from that of  
$v_0\otimes\ket{\emptyset}_{M'}$,   
this contradicts to the assumption that $M$ is $1$-finite. 
Hence we have $\dim V=1$. 
\end{proof}

Next we describe $2$-finite modules.
For $i\in I$ and $a\in \C^\times$, let
\begin{align}\label{N module}
N_{i,a}^+=L(\bm_{i,a}^{(2)}), 
\qquad \bm_{i,a}^{(2)}=
X_{i,a}^{-1}\prod_{j\in I, C_{j,i}\neq 0} X_{j,aq_{j,i}^{-1}}\,.
\end{align}
The properties of the module $N^+_{i,a}$ are given in the following proposition.
\begin{prop}\label{prop:2-finite}
The module $N^+_{i,a}$ is $2$-finite. Moreover we have
\begin{align}
\chi_q(N^+_{i,a})=\bm_{i,a}^{(2)}(1+A^{-1}_{i,a})
\prod_{j: C_{j,i}<0} \bar\chi_j.
\label{qch-N}
\end{align}
\end{prop}
\begin{proof}
Noting that
$\ds{\bm_{i,a}^{(2)}
=Y_{i,aq_i^{-1}} \prod_{j\in I, C_{j,i}<0}X_{j,aq_{j,i}^{-1}}}$,
we set 
\begin{align*}
V=L(Y_{i,aq_i^{-1}})\,,\quad
M=\bigotimes\limits_{j: C_{j,i}<0}M_{j,aq_{j,i}^{-1}}^+\,.
\end{align*}
We use an argument similar to that of the proof of 
Proposition \ref{prop:gradingM}. Let $\bm=Y_{i,aq_i^{-1}}$.
By \eqref{top terms},
the $\ell$-weight subspaces $V_{\bm}$, $V_{\bm A^{-1}_{i,a}}$ 
are both one-dimensional. 
Choose their generators $v_0,v_1$ 
and consider the two dimensional subspace $V^{(0)}=\C v_0+\C v_1$ of $V$. 
We claim that $V^{(0)}\otimes M$ is a submodule of $V\otimes_D M$. 

Indeed, we have $\dim L(\bm)_{\bm A^{-1}_{i,a}}=1$, and 
if $C_{i,j}<0$, then $\dim L(\bm )_{\bm_j}=1$ where  
$\bm_j=\bm A^{-1}_{i,a}A^{-1}_{j,aq_{j,i}^{-1}}$.  
Therefore, by Lemma \ref{delta lem} and   Lemma \ref{polyn act}, we see that 
 $\Delta_D\bigl(x^-_{j,>}(z)\bigr)(v_1\otimes M)\subset V^{(0)}\otimes M$.  
It is also clear that $V^{(0)}\otimes M$ is stable under
 $\Delta_D\bigl(x^+_{j,\gge}(z)\bigr)$. 
Therefore $V^{(0)}\otimes M$ is a submodule by Lemma \ref{other generators lem}.

By Lemma \ref{submodules}, the only possibility for a non-zero proper
submodule of $V^{(0)}\otimes M$ 
is $\C v_0\otimes M$. However, it is easy to see that the latter is not a submodule.
Hence $V^{(0)}\otimes M$ is irreducible. 
Comparing the highest $\ell$-weights, we conclude that 
$N^+_{i,a}=V^{(0)}\otimes M$, from which follows \eqref{qch-N}. 
\end{proof}
The module $N_{i,a}^+$ should be thought of as an analog of $U_q\widehat{\mathfrak{sl}}_2$ module in the direction $i$. 
We will use  modules $N^+_{i,a}$ in Section \ref{transfer sec} to establish Bethe ansatz equations for $XXZ$ 
Hamiltonians in the same way as it is done in the case of $U_q\widehat{\mathfrak{sl}}_2$. For that we will need the following lemma.

\begin{prop}\label{TQ relation} We have the equality in 
the Grothendieck ring $\Rep_F\, U_q\bo$.
\begin{align*}
[N_{i,a}^+][M^+_{i,a}]=
\prod\limits_{j,\ C_{j,i}\neq 0} [M^+_{j,aq_{j,i}^{-1}}]
+\prod\limits_{j,\ C_{j,i}\neq 0} [M^+_{j,aq_{j,i}}].
\end{align*}
\end{prop}
\begin{proof}
The lemma follows from the comparison of $q$-characters.
\end{proof}

More generally, we have the following construction.
Let $V\in \Fin$ be an irreducible $U_q\bo$ module of finite type. 
Then by Corollary \ref{A cor}, 
$\chi_q(V)=\bm(1+\sum_{r=1}^N \bm_r)\prod_{i\in I}\bar\chi_i^{d_i(\bm)}$. 
We call 
$\chi_q^{ess}=1+\sum_{r=1}^N \bm_r\in\Z_{\geq 0}[A^{-1}_{i,a}]_{i\in I,a\in\C^\times}$
the {\it normalized essential $q$-character} of $V$.

Let $J\subset I$. We have the obvious inclusion of rings
$\iota_J:\ \Z[A^{-1}_{i,a}]_{i\in J,a\in\C^\times}\to \Z[A^{-1}_{i,a}]_{i\in I,a\in\C^\times}$. 

\begin{prop}\label{lift prop}
Let $W_J$ be a finite type  $U_q\bo_J$ module. Then there exists a finite type $U_q\bo$ module $W$ such that 
$\iota_J(\chi_q^{ess}(W_J))=\chi_q^{ess}(W)$.
\end{prop}
\begin{proof}
Let $W_J=L(\bm_J)$. Let $V=L(\bm_J)\in \Fin$ be the highest $\ell$-weight $U_q\bo$ module 
where we consider $\bm_J$ as a monomial in $\mc X$. 
We denote $v_0$ 
the highest $\ell$-weight vector of $V$. 
Then $(U_q\bo_J) v_0\simeq W_J$. 
Now similarly to the construction of $N_{i,a}^+$, 
we consider the tensor product $V\otimes_D M$. 
The module $M$ is a tensor product of sufficiently many factors $M_{i,a}^+$, 
$i\not\in J$, so that its highest $\ell$-weight $\bs{\Psi}^M(z)$ satisfies
$\bigl(\Psi_i^M(z) x^-_{i}(z)v\bigr)_{>}=0$ 
for all $i\not\in J$ and 
$v\in V_\bm$ such that 
$\bm$ appears in $\chi_q\bigl(U_q\bo_J)v_0\bigr)$ but
$\bm A_{i,a}^{-1}$ does not appear in $\chi_q\bigl((U_q\bo_J) v_0\bigr)$.

Let $W$ be the irreducible submodule of $V\otimes_D M$ containing 
$v_0\otimes \ket{\emptyset}_M$.
Then the equality 
$\iota_J(\chi_q^{ess}(W_J))=\chi_q^{ess}(W)$ 
follows from Lemma 3.6 in \cite{FM1}.
\end{proof}
The module $W$ in Proposition \ref{lift prop} is not unique, 
it is defined up to tensor multiplication by various $M_{i,a}^+$. 
We denote the $W$ which corresponds to the minimal choice of 
such factors by $W_J^I$ and call 
it the {\it lift of module $W_J$}. Therefore we have
a map $\mc P_J^I:\ \cO^{fin}_{\bo_J}\to \Fin$, $W_J\mapsto W_J^I$. 
Note that it is not a ring homomorphism.

In particular, if we take $J=\{i\}\subset I$, then 
$\g_J\simeq U_q\widehat{\mathfrak{sl}}_2$, 
the fundamental module $L(Y_{a})$ is just a 2-dimensional 
evaluation module and 
we have $\mc P_J^I(L(Y_a))=N^+_{i,aq}$. 

{\it Example.} Let $\g=\widehat{\mathfrak{sl}}_{n+1}$, $J=\{i\}\subset I$. 
Let $W_J=L(\prod_{p=0}^{s-1}Y_{a q^{-2p}})$ 
be the evaluation 
$U_q\widehat{\mathfrak{sl}}_2$ module of dimension $s+1$. In this case formula \eqref{la char} becomes
\begin{align*}
\chi_q(W_J)=
\prod_{p=0}^{s-1}Y_{a q^{-2p}}
\Big(1+\sum_{j=1}^s\prod_{r=1}^jA^{-1}_{aq^{-2r+3}}\Big).
\end{align*}
Then we have 
$\mc P_J^I(W_J)=L(X_{i-1,aq^2}X_{i+1,aq^2}\prod_{p=0}^{s-1}Y_{i,a q^{-2p}})$.

This module is $(s+1)$-finite and it has a similar $q$-character:
\begin{align*}
&\chi_q(L(X_{i-1,aq^2}X_{i+1,aq^2}\prod_{p=0}^{s-1}Y_{i,a q^{-2p}}))
\\
&=X_{i-1,aq^2}X_{i+1,aq^2}(\prod_{p=0}^{s-1}Y_{i,a q^{-2p}})
\Big(1+\sum_{j=1}^s\prod_{r=1}^jA^{-1}_{i,aq^{-2r+3}}\Big)\bar\chi_{i-1}\bar\chi_{i+1}.
\end{align*}

{\it Example.} Let $\g=\widehat{\mathfrak{sl}}_{3}$. Consider $L(Y_{1,aq^{-2}}Y_{1,a})$. 
It is the $6$-dimensional evaluation module related to partition $\la=(2\geq 0)$, see \eqref{la char}. 
The module $L(Y_{1,aq^{-2}}Y_{1,a}X_{2,aq^2})$ is the $3$-finite module which is a lift of the $3$-dimensional 
$U_q\widehat{\mathfrak{sl}}_{2}$ evaluation module $L(Y_{aq^{-2}}Y_a)$. The module $L(Y_{1,aq^{-2}}Y_{1,a}X_{2,a})$ 
is a $5$-finite module and it is not a lift of a $U_q\widehat{\mathfrak{sl}}_{2}$ module.

\section{Bethe Ansatz}\label{transfer sec}
\subsection{Normalized $R$ matrix}\label{sec:normalized-R}

In this section, we take $u$ to be an indeterminate. 

We define $s_u:\U_q\g\to U_q\g[u^{\pm1}]$ by setting 
$s_u(x)=u^{\hdeg x}x$ for any homogeneous element  $x\in U_q\g$.
We set also 
\begin{align*}
 \cR(u)=\bigl(s_u\otimes\id\bigr)\bigl(\cR\bigr)\quad
\in U_q\bo\widehat{\otimes}U_q\bbo [[u]]\,.
\end{align*}
In formula \eqref{R+} for $\cR_+$,  
the first tensor component of each term acts as 
annihilation operator on modules from $\cO_{\bo}$.
Likewise, in formula \eqref{R-} for $\cR_-$,  
the second tensor component of each term acts as
annihilation operator on modules from $\cO_{\g}$. 
Therefore, if $V\in \Ob\cO_{\bo}$ and $W\in \Ob\cO_\g$, then
each coefficient of the formal series $\cR(u)$ 
is a well defined operator on $V\otimes W$.  

Suppose further that $V$ is a tensor product of highest $\ell$-weight $U_q\bo$ modules,   
and  $W$ is a tensor product of highest $\ell$-weight $U_q\g$ modules.  
Denote by $v_0\in V$ the tensor product of highest $\ell$-weight vectors, and 
by $w_0\in W$ the tensor product of highest $\ell$-weight vectors. 
We write the eigenvalues of $h_{i,r}$ on these vectors as $\br{h_{i,r}}_V$, $\br{h_{i,r}}_W$, respectively. 
From the remark above, we see that 
$\cR(u)\bigl(v_0\otimes w_0\bigr)=f_{V,W}(u)\bigl(v_0\otimes w_0\bigr)$, 
where
\begin{align}
f_{V,W}(u)=q^{-(\wt v_0,\wt w_0)}
\exp\Bigl(-\sum_{r>0\atop i,j\in I}\frac{r\widetilde{B}_{i,j}(q^r)}{q^r-q^{-r}}
(q_i-q_i^{-1})(q_j-q_j^{-1})
\, u^r \br{h_{i,r}}_V\br{h_{j,-r}}_W\Bigr)\,.
\label{fVW}
\end{align}
We have
\begin{align*}
f_{V_1\otimes V_2,W}(u)=f_{V_1,W}(u)f_{V_2,W}(u)\,,
\quad
f_{V, W_1\otimes W_2}(u)=f_{V,W_1}(u)f_{V,W_2}(u)\,.
\end{align*}

We define the normalized $R$ matrix 
$\Rb_{V,W}(u)\in \End\bigl(V\otimes W\bigr)[[u]]$ by 
\begin{align*}
\Rb_{V,W}(u)\bigl(v\otimes w\bigr)=
f_{V,W}(u)^{-1}\cR(u)\bigl(v\otimes w\bigr)\quad
(v\in V, w\in W).
\end{align*}

\begin{lem}
Let $V=M^+_{i,1}$, and let 
$W\in\Ob\cO_\g$ be an irreducible $U_q\g$ module. 
Then the normalized $R$ matrix $\Rb_{M^+_{i,1},W}(u)$
does not have a pole in $u\in\C^{\times}$.
\end{lem}
\begin{proof}
Recall that $\Rb_{M^+_{i,1},W}(u)(v_0\otimes w_0)=v_0\otimes w_0$, 
where $v_0\in M^+_{i,1}$, $w_0\in W$ are the highest $\ell$-weight
vectors. 
Suppose that $\Rb_{M^+_{i,1},W}(u)$ has a pole at $u=u_0$. 
Let $P:V\otimes W\to W\otimes V$ be the map $P(v\otimes w)=w\otimes v$.
Then $\res_{u=u_0}P\,\Rb_{M^+_{i,1},W}(u): 
V\otimes W(u^{-1}_0)\to W(u^{-1}_0)\otimes V$ 
is an intertwiner of  $U_q\bo$ modules,
where $W(u^{-1}_0)$ is the module $W$ twisted by $s^{-1}_{u_0}$. 
Its image is a non-zero submodule of $W(u_0^{-1})\otimes M^+_{i,1}$,
which does not contain $w_0\otimes v_0$ due 
to the normalization of $\Rb_{M^+_{i,1},W}(u)$.  
However Lemma \ref{submodules} shows that there is no such submodule. 
This is a contradiction. 
\end{proof}

\subsection{Transfer matrices} 

Let $\tilde{p}=(\tilde{p}_i)_{i\in I}$ be indeterminates, and 
set $p=(p_i)_{i\in I}$, $p_i=\prod_{j\in I}\tilde{p}_j^{C_{j,i}}$. 
For $V\in\Ob\cO_\bo$, denote by $\tilde{p}^h\in \End V$ the operator
which acts on each weight vector $v\in V_\mu$ by
$\tilde{p}^h v=(\prod_{i\in I}\tilde{p}_i^{\log\mu_i/\log q_i})v$. 

For an object $V\in \Ob\cO_{\bo}$, 
the twisted transfer matrix 
associated with the `auxiliary space' $V$ 
is a formal series defined by
\begin{align*}
\Tb_V(u;p)&=\Tr_{V,1}\left(\tilde{p}^{-h}\otimes \id\cdot \cR(u)\right)
\quad \in U_q\bbo[\tilde{p}_i^{\pm1}][[u,p_i]]_{i\in I}\,.
\end{align*}
Here $\Tr_{V,1}$ means that the trace is taken on the first tensor component.
Clearly we have 
\begin{align*}
&\Tb_{V_1\oplus V_2}(u;p)=\Tb_{V_1}(u;p)+\Tb_{V_2}(u;p)\,,
\\
&\Tb_{V_1\otimes V_2}(u;p)=\Tb_{V_1}(u;p)\Tb_{V_2}(u;p)\,,
\end{align*}
hence the assignment
$V\mapsto \Tb_V$ gives a homomorphism of rings from 
$\Rep\,U_q\bo$ to $U_q\bbo[\tilde{p}_i^{\pm1}][[u,p_i]]_{i\in I}$.

Element $\Tb_V(u;p)$ gives rise to 
a formal series of operators which act 
on any given `quantum space' $W\in\Ob\cO_\g$.  
It is convenient to use the normalized $R$ matrix and define
\begin{align*}
T_{V,W}(u;p)
&=\Tr_{V,1}\bigl((\tilde{p}^{-h}\otimes\id)\Rb_{V,W}(u)\bigr)\,
\quad \in \End(W)[\tilde{p}_i^{\pm1}][[u,p_i]]_{i\in I}\,,
\end{align*}
so that $\Tb_V(u;p)\bigl|_{W}=f_{V,W}(u)T_{V,W}(u;p)$. 
Note that $T_{V,W}(u;p)$ acts on each subspace of $W$ of fixed 
weight.

\subsection{Bethe Ansatz}

From now on, we choose $W$ to be a tensor product of irreducible 
finite dimensional $U_q\g$ modules. 
Let $w_0\in W$ be the tensor product of highest 
$\ell$-weight vectors, with weight $\mu=\wt w_0$ and highest $\ell$-weight
$\bm=\prod_{i\in I,b\in\C^\times}Y_{i,b}^{m_{i,b}}$. 
Then $\phi^+_i(z)w_0=\prod_{b\in\C^\times} 
((q_i-bz)/(1-q_ibz))^{m_{i,b}}w_0$.
We set 
\begin{align*}
&a_i(u)=\prod_{b}\Bigl(q_iu-b\Bigr)^{m_{i,b}}\,,
\quad
d_i(u)=\prod_{b}\Bigl(u-q_ib\Bigr)^{m_{i,b}}\,,
\end{align*}
and introduce the notation 
\begin{align*}
&Q_i(u;p)=T_{M^+_{i,1},W}(u;p)\,,
\quad
\cT_i(u;p)=a_i(u)T_{N^+_{i,1},W}(u;p)\,.
\end{align*}

The following is an analog of Baxter's relation well known for $\g=\slth$.
\begin{lem}\label{lem:TQrel}
We have the relation
\begin{align}
\cT_i(u;p)Q_i(u;p)&=a_i(u) \prod_{j:C_{j,i}\neq0} Q_j(q_{j,i}^{-1}u;p)
+
p_i d_i(u) \prod_{j:C_{j,i}\neq0} Q_j(q_{j,i} u;p)\,.
\label{TQ-rel}
\end{align}
\end{lem}
\begin{proof}
This follows from Proposition \ref{TQ relation}
and the following calculation:
\begin{align*}
\frac{d_i(u)}{a_i(u)}&=
\prod_{j:C_{j,i}\neq0}\frac{f_{M^+_{j,1},W}(q_{j,i}u)}{f_{M^+_{j,1},W}(q^{-1}_{j,i}u)}\\
&=q_i^{(\alpha_i,\wt w_0)}\exp\Bigl(\sum_{l,j\in I\atop r>0}\frac{q_{i,l}^r-q_{i,l}^{-r}}{q^r-q^{-r}}\tilde{B}_{l,j}(q^r)(q_j-q_j^{-1})
\br{h_{j,-r}}_Wu^r\Bigr)
\\
&=\br{\phi^-_i(u^{-1})^{-1}}_W\,.
\end{align*}
\end{proof}

We show next that on each weight subspace of $W$ the operators
$Q_i(u;p)$, $\cT_i(u;p)$ are polynomials in $u$. 
Since they are defined as traces on infinite dimensional spaces,  
we need certain estimates on the growth of the matrix elements to ensure the polynomiality.  

For vectors $w^*\in W^*$ and $w\in W$, introduce the 
notation $L_{w^*,w}(u)$ for the matrix coefficients in the second component, 
\begin{align*}
v^*L_{w^*,w}(u)v
=v^*\otimes w^*\Rb_{V,W}(u)
v\otimes w
\quad (v^*\in V^*, v\in V).
\end{align*} 
We regard $V^*$, $W^*$ as right $U_q\bo$ modules. 
The intertwining property of the $R$ matrix
and Lemma \ref{lem:copro} entail the following relations. 
\begin{align}
L_{w^*x^+_{i,k},w}(u)&=
L_{w^*,x^+_{i,k}w}(u)k_i
+u^k(L_{w^*,w}(u)x^+_{i,k}-x^+_{i,k}L_{w^*k_i,w}(u))
\label{Lxp}
\\
&+\sum_j u^{\hdeg a_j} L_{w^*,b_jw}(u)a_j
-\sum_j u^{\hdeg b_j} b_jL_{w^*a_j,w}(u)\,,
\nn\\
L_{w^*,x^-_{i,k}w}(u)&=
k_i^{-1}L_{w^*x^-_{i,k},w}(u)
-u^k(L_{w^*,k_i^{-1}w}(u)x^-_{i,k}-x^-_{i,k}L_{w^*,w}(u))
\label{Lxm}
\\
&-\sum_j u^{\hdeg a'_j} L_{w^*,b'_jw}(u)a'_j
+\sum_j u^{\hdeg b'_j} b'_jL_{w^*a'_j,w}(u)\,.
\nn
\end{align}
Here $a_j\in U_q\bo$, $b_j\in U^+_q\bo$ are such that 
$\wt\, b_j>0$, $\wt\, a_j+\wt\, b_j=\alpha_i$, 
and $a'_j\in U^-_q\bo$, $b'_j\in U_q\bo$ 
are such that $\wt\, a'_j<0$,
$\wt\, a'_j+\wt\, b'_j=-\alpha_i$.
Even though $U_q^-\bo$ is not generated by the $x^-_{j,k}$,  
they are enough to generate the space $W$ from $w_0$.  

\begin{prop}\label{Q pol}
On each weight subspace $W_{\mu \sfq^{-\beta}}$ ($\beta\in Q^+$), $Q_i(u;p)$ is a polynomial in $u$
of degree at most $(\omega^\vee_i,\beta)$.  
\end{prop}
\begin{proof}
Set $V=M^+_{i,a}$, 
and let $V=\oplus_{m=0}^\infty M[m]$ 
be the grading as in Proposition \ref{prop:gradingM}.   
Denote by $u^\partial$ the operator $u^\partial\bigl|_{M[m]}=u^m\times\id_{M[m]}$. 
The assertion is proved if we show that as $u\to\infty$ 
\begin{align}
u^{-\partial} L_{w^*,w}(u)u^{\partial} =
O(u^{(\omega_i^\vee,\beta)})\qquad (w^*\in W^*_{\mu\sfq^{-\beta}},\ w\in W).
\label{growthM}
\end{align}

We set $\bar h_{i,r,V}=h_{i,r}-\br{h_{i,r}}_V$.
Proposition \ref{prop:gradingM} implies that the operators  
 $u^ru^{-\partial}\bar h_{i,r,V}u^\partial$ and  
$u^{\hdeg x}u^{-\partial} x u^\partial$ ($x\in U^-_q\bo$) 
are independent of $u$, while 
\[
u^{\hdeg x}u^{-\partial} x u^\partial=O(u^{(\omega_i^\vee,\wt\, x)}) 
\quad (x\in U^+_q\bo).
\]

Consider the case where $w^*=w_0^*$ is the lowest weight vector
and $w=w_0$ is the highest weight vector. Then we have
\begin{align*}
L_{w_0^*,w_0}(u)=
\exp\Bigl(-\sum_{r>0\atop k,j\in I}\frac{r\widetilde{B}_{k,j}(q^r)}{q^r-q^{-r}}
(q_k-q_k^{-1})(q_j-q_j^{-1})
\, u^r \bar h_{k,r,V}\br{h_{j,-r}}_W\Bigr)\,.
\end{align*}
From the remark above we see that
$u^{-\partial}L_{w_0^*,w_0}(u)u^\partial$ is independent of $u$. 
Hence \eqref{growthM} holds in this case. 

By induction on the height of $\beta$, and using \eqref{Lxm},
we can show that $u^{-\partial}L_{w_0^*,w_2}(u)u^\partial=O(1)$ for all $w\in W$. 
Finally \eqref{growthM} in the general case follows from \eqref{Lxp}. 
\end{proof}

\begin{prop}\label{T pol}
On  each weight subspace of $W$, $\cT_i(u;p)$ is a polynomial in $u$.
\end{prop}
\begin{proof}
Set $V=N_{i,a}^+$, and let $V=N_0\oplus N_1$ 
be the decomposition into $\ell$-weight subspaces. 
In the notation of the proof of Proposition \ref{prop:2-finite}, we have 
$N_s=\C v_s\otimes M \subset V^{(0)}\otimes_D M$, $M=\bigotimes_{j:C_{j,i}<0}M^+_{j,aq_{j,i}^{-1}}$.
Set $\omega^\vee=\sum_{j:C_{j,i}<0}\omega^\vee_j$. 
Using the grading $M=\oplus_{m=0}^\infty M[m]$ of $M$ in Proposition \ref{prop:gradingM}
we set $N_s[m]=v_s\otimes M[m]$. Let $u^\partial$ be the operator which act as scalar $u^m$ on 
$N_j[m]$. 
We have then
\begin{align*}
L_{w_0^*,w_0}(u)
&=
\exp\Bigl(-\sum_{r>0\atop k,j\in I}\frac{r\widetilde{B}_{k,j}(q^r)}{q^r-q^{-r}}
(q_k-q_k^{-1})(q_j-q_j^{-1})
\, u^r \bar h_{k,r,N_0}\br{h_{j,-r}}_W\Bigr)\,
\\
&=\frac{d_i(u)}{a_i(u)}
\exp\Bigl(-\sum_{r>0\atop k,j\in I}\frac{r\widetilde{B}_{k,j}(q^r)}{q^r-q^{-r}}
(q_k-q_k^{-1})(q_j-q_j^{-1})
\, u^r \bar h_{k,r,N_1}\br{h_{j,-r}}_W\Bigr)\,.
\end{align*}
Hence $a_i(u)L_{w_0^*,w_0}(u)$ is a polynomial on each vector.

Acting with $\Delta_D(x^\pm_{j,k})$ we find 
\begin{align*}
&u^k u^{-\partial}x^-_{j,k}u^{\partial} N_s[m]\subset N_s[m+k]+
\delta_{s,0}\delta_{i,j}\sum_{l\ge0}u^{k-l}N_1[m+l]\,,
\\
&u^k u^{-\partial}x^+_{j,k}u^{\partial} N_s[m]
\subset \sum_{r=0}^{(\omega^\vee,\alpha_j)}u^r \sum_{l\ge k}u^{k-l}N_s[m+l-r]+
\delta_{s,1}\delta_{i,j}u^{k}N_0[m]\,.
\end{align*}
Note that we need only a finitely many of the $x^{\pm}_{j,k}$'s
in order to generate $W$ (resp. $W^*$) from $w_0$ (resp. $w_0^*$). 
The rest of the argument is the same as in the proof of 
Proposition \ref{Q pol}; we use the intertwining relations \eqref{Lxp}, \eqref{Lxm}
and induction on the weight to prove that $u^{-\partial}L_{w^*,w}(u)u^\partial =O(u^{K})$ 
for some $K=K_{w^*,w}$. 
\end{proof}
\medskip

Let $w$ be an eigenvector of $Q_i(u;p)$ with eigenvalue $Q_{i,w}(u;p)$. 
By Proposition \ref{Q pol}, we can write
$Q_{i,w}(u;p)=Q_{i,w}(0;p)\prod_{\nu=1}^{N_i}(1-u/\zeta_{i,\nu})$. 
Substituting $u=\zeta_{i,\nu}$ into \eqref{TQ-rel}, and using the polynomiality of $\cT_i(u;p)$, 
we obtain the Bethe ansatz equations
\begin{align}
p_i\frac{d_i(\zeta_{i,\nu})}{a_i(\zeta_{i,\nu})}
\prod_{j\in I}
\prod_{\mu=1}^{N_j}\frac{1-q_{j,i}\zeta_{i,\nu}/\zeta_{j,\mu}}
{1-q_{j,i}^{-1}\zeta_{i,\nu}/\zeta_{j,\mu}}
=-1
\quad (i\in I, \nu=1,\cdots,N_i).
\label{BAE}
\end{align}

The corresponding eigenvalue of the normalized transfer matrix 
can be obtained from the $q$-character of the `auxiliary space' $V$.
The recipe is given as follows \cite{FH}.
\begin{thm}\label{T-eigv}
Let $w$  be an eigenvector of $T_{V,W}(u;p)$ of weight 
$(q_i^{\nu_i})$. 
Then the corresponding eigenvalue of $f_{V,W}(u;p)T_{V,W}(u;p)$
is obtained from $\chi_q(V)$ by the substitution 
\[
X_{i,a} \to f_{M^+_{i,1},W}(a;p)Q_{i,w}(a;p)\,,\quad
y_i^{b_i} \to q^{-(b_i\omega_i,\sum_{j}\nu_j\omega_j)}.
\]
\end{thm}
\appendix

\section{Root vectors}
 
Following \cite{Be,Be2,Da}, 
we review the definition and 
known facts about root vectors of $U_q\g$.
In this section we take $q$ to be an indeterminate and work over 
$\C(q)$. 
\medskip

Let 
$\Omega, \Xi:U_q\g\to U_q\g$ be anti-isomorphisms of $\C$-algebras 
defined by
\begin{align*}
&\Omega e_i=f_i\,,\quad \Omega f_i=e_i\,,\quad \Omega k_i=k^{-1}_i\,
\quad 
(0\le i\le n),\quad \Omega q=q^{-1},\\
&\Xi e_i=e_i\,,\quad \Xi f_i=f_i\,,\quad \Xi k_i=k^{-1}_i\,
\quad (0\le i\le n),\quad \Xi q=q\,.
\end{align*}
Denote by  $s_i$ ($0\le i\le n$) the simple reflections. 
For $\alpha=\sum_{i\in I}b_i\alpha_i\in Q$ 
we set $k_\alpha=\prod_{i\in I}k_i^{b_i}$. 
Lusztig's automorphisms $\{T_i\}_{0\le i\le n}$ of $U_q\g$ are 
characterized by the following properties: 
\begin{align*}
&T_ik_\alpha=k_{s_i\alpha}\quad (\alpha\in Q)\,,\\
&T_ie_j=
\begin{cases}
-f_ik_i & (i=j),
\\
\sum_{r+s=-C_{i,j}}(-1)^sq_i^{-r}e_i^{(s)}e_je_i^{(r)}
&(i\neq j),	\\
\end{cases}
\\
&\Omega T_i=T_i \Omega\,,\quad \Xi T_i=T_i^{-1}\Xi\,.
\end{align*}
For an element $w$
of the affine Weyl group $W=\langle s_i\ (0\le i\le n)\rangle$, 
we define $T_w=T_{i_1}\cdots T_{i_m}$
where $w=s_{i_1}\cdots s_{i_m}$ is a reduced expression. 
This definition does not depend on the chosen expression of $w$.  

There is an embedding of groups
$t: Q^\vee=\sum_{i\in I}\Z\alpha^\vee_i\to W$, $x^\vee\mapsto t_{x^\vee}$, 
such that $t_{x^\vee}(\alpha)=\alpha-(x^\vee,\alpha)\delta$
($\alpha\in Q$).
Fix an element $x^\vee\in  Q^\vee$
which satisfies $(x^\vee,\alpha_i)>0$ for all $i\in I$.
Fix also a reduced decomposition $t_{x^\vee}=s_{i_1}\cdots s_{i_N}$
such that $i_1=0$,  
and introduce a sequence $\{i_r\}_{r\in \Z}$ 
by setting $i_{r+N}=i_r$ for $r\in \Z$. 
We define $\{\beta_r\}_{r\in \Z}$ by
\begin{align*}
\beta_r=\begin{cases}
s_{i_1}\cdots s_{i_{r-1}}\alpha_{i_r}& (r\ge1),\\
s_{i_0}\cdots s_{i_{r+1}}\alpha_{i_r}& (r\le0).\\
\end{cases} 
\end{align*}
By construction we have $\beta_{r+N}=\beta_r-(x^\vee,\beta_r)\delta$. 
For any $s\ge1$ the following hold \cite{Be2}:
\begin{align}
&\{\beta_r\mid 1\le r\le sN\}
=\{m\delta-\alpha\mid \alpha\in \Delta_+\,,\
1\le m\le s(x^\vee,\alpha)\} \,,
\label{real-rootv1}\\
&\{\beta_r\mid 0\ge r\ge -sN+1\}
=\{m\delta+\alpha\mid \alpha\in \Delta_+\,,\
0\le m\le s(x^\vee,\alpha)-1\} \,.
\label{real-rootv2}
\end{align}
In particular, $\{\beta_r\}_{r\in \Z}$ 
coincides with the set of positive real roots of $\g$. 

The root vectors are defined 
for positive real roots by
\begin{align}
&e_{\beta_r}
=\begin{cases}
T_{i_1}\cdots T_{i_{r-1}} e_{i_r}& (r\ge1)\,,\\
T_{i_0}\cdots T_{i_{r+1}} e_{i_r}& (r\le0)\,,\\
\end{cases}
\label{rootv-1}
\\
&f_{\beta_r}=\Omega e_{\beta_r}\,. 
\label{rootv-2}
\end{align}
Positive imaginary roots have multiplicities. Labeling them  
as $(m\delta,i)$ ($m>0$, $i\in I$) we set 
\begin{align}
&e_{(m\delta,i)}=q_i^{-2}e_ie_{m\delta-\alpha_i}
-e_{m\delta-\alpha_i}e_i\quad (m>0)\,,\quad
f_{(m\delta,i)}=\Omega e_{(m\delta,i)}\,.
\label{rootv-3}
\end{align}
The root vectors are related to the Drinfeld generators by
\begin{align}
&x^+_{i,m}=o(i)^m 
\times\begin{cases}
e_{m\delta+\alpha_i}& (m\ge0),\\
-f_{-m\delta-\alpha_i}k_i^{-1}& (m<0),\\
\end{cases}
\label{x+Dri}\\
&\phi^+_{i,m}=-o(i)^m (q_i-q_i^{-1})k_i e_{(m\delta,i)}\ (m>0)\,,
\label{phiDri}\\
&x^-_{i,m}=\Omega x^+_{i,-m}\,,
\quad \phi^-_{i,r}=\Omega \phi^+_{i,-r}\,.
\label{x-Dri}
\end{align}
Here $o(i)=\pm1$ is chosen 
in such a way that $o(i)o(j)=-1$ holds for $C_{i,j}<0$. 
The Borel subalgebras are generated by root vectors, 
\begin{align*}
&U^+_q\bo=\langle e_{m\delta+\alpha}\mid m\ge0,\ \alpha\in\Delta^+ \rangle\,,
\\
&U^-_q\bo=\langle k_{\alpha}
e_{m\delta-\alpha}\mid m>0,\ \alpha\in\Delta^+ \rangle
\,.
\end{align*}

Define a total order $\prec$ on the set of positive roots as follows. 
For real roots we set $\beta_r\prec\beta_{s}$ if $r>s>0$ or $s<r\le 0$. 
Positive imaginary roots 
are mutually ordered in an arbitrary way. We set further
$\beta_s\prec (m\delta,i)\prec \beta_{r}$ for any $m>0$, $r>0\ge s$.
Altogether we have
\begin{align}
\beta_0\prec\beta_{-1}\prec\beta_{-2}\prec\cdots\prec
(m\delta,i)\prec \cdots
\prec\beta_3\prec\beta_2\prec\beta_1\,. 
\label{root-order}
\end{align}
The root vectors satisfy the convexity property \cite{Be2}
\begin{align}
e_\beta e_\alpha -q^{(\alpha,\beta)}e_\alpha e_\beta 
=\sum_{\{\gamma_i\},\{n_i\}}a^{\{n_i\}}_{\{\gamma_i\}}
e_{\gamma_1}^{n_1}\cdots e_{\gamma_m}^{n_m}\,,
\label{convex}
\end{align}
where $a^{\{n_i\}}_{\{\gamma_i\}}\in\C(q)$, and 
the sum is taken over $\gamma_i$ and $n_i\in\Z_{>0}$ such that
$\alpha\prec\gamma_1\prec\cdots\prec\gamma_m\prec\beta$, 
$\sum_{i}n_i\gamma_i=\alpha+\beta$. 
If $\alpha=\beta_r$, $\beta=\beta_s$ and  $r,s>0$ or $r,s\le0$, then the coefficients are Laurent polynomials of $q$
\cite{BCP}. If in addition $\alpha+\beta$ is a root, then $e_{\alpha+\beta}$ appears with non-zero coefficient.
\begin{lem}\label{lem:root-vec}
For any $\alpha\in \Delta^+$ and $l\ge1$, there exists
a $p_0\ge1$ such that if $p\ge p_0$ then
$e_{p\delta\pm\alpha}$ belongs to the algebra 
generated by $x^\pm_{i,m}$ and $k_i^{\pm1}$
($i\in I$, $m\ge l$). 
\end{lem}
\begin{proof}
This can be verified 
by induction on the height of $\alpha$, 
using \eqref{real-rootv1}, \eqref{real-rootv2}, and
\eqref{convex}.
\end{proof}

The following can be extracted from \cite{Da}, Theorem 4:
\begin{lem}\label{lem:copro}
We have
\begin{align*}
\Delta(x^+_{i,k})=x^+_{i,k}\otimes 1+k_i\otimes x^+_{i,k}
+\sum_{j}a_j\otimes b_j\,,
\end{align*}
where $a_j\in U_q\bo$, $b_j\in U^+_q\bo$ 
are such that 
$\wt\, b_j>0$, $\wt\, a_j+\wt\, b_j=\alpha_i$.
Similarly
\begin{align*}
\Delta(x^-_{i,k})=x^-_{i,k}\otimes k_i^{-1}+1\otimes x^-_{i,k}
+\sum_{j}a'_j\otimes b'_j\,,
\end{align*}
where $a'_j\in U^-_q\bo$, $b'_j\in U_q\bo$ 
are such that $\wt a'_j<0$,
$\wt a'_j+\wt b'_j=-\alpha_i$.
\end{lem}
\bigskip

{\bf Acknowledgments.}
MJ wishes to thank David Hernandez for stimulating discussions. 

The research of BL is supported by 
the Russian Science Foundation grant project 16-11-10316. 
MJ is partially supported by 
JSPS KAKENHI Grant Number JP16K05183. 
EM is partially supported by a grant from the Simons Foundation  
\#353831.

EM and BF would like to thank Kyoto University
for hospitality during their visits when this work was started. 
EM is grateful to Masaki Kashiwara for supporting his visit.

\end{document}